\documentclass[
10pt, % Main document font size
a4paper, % Paper type, use 'letterpaper' for US Letter paper
oneside, % One page layout (no page indentation)
headinclude,footinclude, % Extra spacing for the header and footer
BCOR = 5mm, % Binding correction
]{scrartcl}

\usepackage[
nochapters, % Turn off chapters since this is an article        
pdfspacing, % Makes use of pdftex’ letter spacing capabilities via the microtype package
dottedtoc % Dotted lines leading to the page numbers in the table of contents
]{classicthesis} % The layout is based on the Classic Thesis style

\usepackage[left=1in,right=1in,top=1in,bottom=2in]{geometry}

\usepackage[T1]{fontenc} % Use 8-bit encoding that has 256 glyphs

\usepackage[utf8]{inputenc} % Required for including letters with accents

\usepackage{graphicx} % Required for including images
\graphicspath{{Figures/}} % Set the default folder for images

\usepackage{enumitem} % Required for manipulating the whitespace between and within lists

\usepackage{lipsum} % Used for inserting dummy 'Lorem ipsum' text into the template

\usepackage{subfig} % Required for creating figures with multiple parts (subfigures)

\usepackage{amsmath,amssymb,amsthm,tikz-cd, mathtools} % For including math equations, theorems, symbols, etc

\usepackage{varioref} % More descriptive referencing

\theoremstyle{definition} % Define theorem styles here based on the definition style (used for definitions and examples)
\newtheorem{definition}{Definition}

\theoremstyle{plain} % Define theorem styles here based on the plain style (used for theorems, lemmas, propositions)
\newtheorem{theorem}{Theorem}
\newtheorem*{theorem*}{Theorem}
\newtheorem{lemma}[theorem]{Lemma}

\newtheorem*{corollary*}{Corollary}

\theoremstyle{remark} % Define theorem styles here based on the remark style (used for remarks and notes)

\newtheorem{remark}{Remark}
\newtheorem{example}{Example}
\newtheorem{question}{Question}
\newtheorem*{notation}{Notation}

\hypersetup{
colorlinks=true, breaklinks=true, bookmarks=true,bookmarksnumbered,
urlcolor=webbrown, linkcolor=RoyalBlue, citecolor=webgreen, % Link colors
pdftitle={}, % PDF title
pdfauthor={\textcopyright}, % PDF Author
pdfsubject={}, % PDF Subject
pdfkeywords={}, % PDF Keywords
pdfcreator={pdfLaTeX}, % PDF Creator
pdfproducer={LaTeX with hyperref and ClassicThesis} % PDF producer
}

\newcommand{\Id}{\operatorname{I}}
\newcommand{\PP}{\mathcal P}
\newcommand{\md}{\ell^{n}}

\DeclareMathOperator{\Aut}{\mathrm{Aut}}
\DeclareMathOperator{\tr}{\mathrm{tr}}

\DeclareMathOperator{\GL}{\mathrm{GL}}

\title{\normalfont{On the variation of Frobenius eigenvalues in a skew-abelian Iwasawa tower}} % The article title

\author{\spacedlowsmallcaps{Asvin G}\footnote{\textit{Department of Mathematics; University of Wisconsin, Madison.}}
\footnote{\textit{  email: gasvinseeker94@gmail.com}}} % The article author(s) - author affiliations need to be specified in the AUTHOR AFFILIATIONS block

\date{} % An optional date to appear under the author(s)

\begin{document}

%----------------------------------------------------------------------------------------
%	HEADERS
%----------------------------------------------------------------------------------------

\renewcommand{\sectionmark}[1]{\markright{\spacedlowsmallcaps{#1}}} % The header for all pages (oneside) or for even pages (twoside)
\lehead{\mbox{\llap{\small\thepage\kern1em\color{halfgray} \vline}\color{halfgray}\hspace{0.5em}\rightmark\hfil}} % The header style

\pagestyle{scrheadings} % Enable the headers specified in this block

%----------------------------------------------------------------------------------------
%	TABLE OF CONTENTS & LISTS OF FIGURES AND TABLES
%----------------------------------------------------------------------------------------

\maketitle % Print the title/author/date block

\setcounter{tocdepth}{2} % Set the depth of the table of contents to show sections and subsections only

%----------------------------------------------------------------------------------------
%	ABSTRACT
%----------------------------------------------------------------------------------------

\begin{abstract}
We study towers of varieties over a finite field such as $y^2 = f(x^{\ell^n})$ and prove that the characteristic polynomials of the Frobenius on the \'etale cohomology show a surprising $\ell$-adic convergence. We prove this by proving a more general statement about the convergence of certain invariants related to a skew-abelian cohomology group. The key ingredient is a generalization of Fermat's Little Theorem to matrices. Along the way, we will prove that many natural sequences of polynomials $(p_n(x))_{n\geq 1} \in \mathbb Z_\ell[x]^{\mathbb N}$ converge $\ell$-adically and give explicit rates of convergence.
    
\end{abstract} % This section will not appear in the table of contents due to the star (\section*)

\tableofcontents % Print the table of contents

%----------------------------------------------------------------------------------------
%	AUTHOR AFFILIATIONS
%----------------------------------------------------------------------------------------

%----------------------------------------------------------------------------------------

\newpage % Start the article content on the second page, remove this if you have a longer abstract that goes onto the second page

%----------------------------------------------------------------------------------------
%	INTRODUCTION
%----------------------------------------------------------------------------------------

\begin{notation}
We will work throughout over a fixed finite field $\mathbb F_q$. A curve $C$ over $\mathbb F_q$ refers to a smooth, projective, geometrically connected scheme of dimension $1$. The base change to the algebraic closure $\overline{\mathbb F}_q$ is denoted by $\overline{C}$. We denote its \'etale cohomology with $\mathbb Z_\ell$ coefficients by $H^1_{\text{\'et}}(\overline{C},\mathbb Z_\ell)$. By standard functoriality arguments, it comes endowed with a linear action of the geometric Frobenius $\sigma_q$. We fix an auxiliary prime $\ell$ throughout and for simplicity assume that $\ell > 2$ and $q\equiv 1 \pmod{\ell}$. \footnote{As usual, the theorems go through if $\ell=2$ with appropriately stronger hypothesis. For instance, if $\ell=2$ then we need $q \equiv 1 \pmod{\ell^2}$.}
\end{notation}

\section{Introduction}

The eigenvalues of the Frobenius on the \'etale cohomology of a smooth, projective variety over a finite field carry significant arithmetic information. By the Weil conjectures, these eigenvalues are algebraic integers and their absolute values under any complex embedding are understood.

We draw inspiration from Iwasawa theory to study the asymptotic behaviour of these eigenvalues in an "Iwasawa tower" and in particular, we show that there is a strong $\ell$-adic convergence statement to be made in many natural examples. The Iwasawa algebras arising in this study are non-commutative due to the non trivial action of the Frobenius on this monodromy group and we hope that this perspective is interesting too. Let us begin with an example.

\begin{example}\label{eg: motivating example}

Consider the smooth projective curves $C_n$ corresponding to the equations
\[Y^2 = X^{2^n} + 1 \text{ over } \mathbb F_5.\]
They define a tower $\dots\to C_2\to C_1$ with maps $C_{n+1} \to C_n$ defined by $(X,Y) \to (X^2,Y)$. The characteristic polynomial of $\sigma_5$ on $H^1_{\text{\'et}}(\overline{C}_n,\mathbb Z_\ell)$ is
\[f_n(x) \coloneqq \det\left(1-\sigma_2x|H^1_{\text{\'et}}(\overline{C}_n,\mathbb Z_\ell)\right) = (1-2x+5x^2)\prod_{i=1}^{n-2}(1 + x^{2^{i}}5^{2^{i-1}})^2.\]
Note that $f_{n-1}(x)$ divides $f_n(x)$ and the inverse of the roots of $g_n(x) = f_n(x)/f_{n-1}(x)$ are of the form $\sqrt{5}\zeta$ for $\zeta$ a root of unity of order $2^{n-1}$ for $n \geq 3$. In Section \ref{sec: examples}, we show that for $n$ sufficiently large, the normalized (by $\alpha \to \alpha/|\alpha|$ so that the complex absolute value is $1$\footnote{we note that the complex norm $|\alpha|$ is independent of the embedding to $\mathbb C$ by the Weil conjectures}) roots of $g_{n+1}(x)$ are exactly all possible $\ell$-th roots of the normalized roots of $g_n(x)$. 
\end{example}

In fact, we prove the same statement for towers of Fermat curves (from which the above follows) and Artin-Schreier curves. The proof of this statement follows from realizing the roots of $g_n(x)$ as Jacobi sums and using results of Coleman \cite{Coleman} on identities for Gauss sums (coming from the Gross-Kubota p-adic Gamma function \cite{Gross-Koblitz}).

\subsection{A congruence on characteristic polynomials}

This prompts the question of what happens in a more general context. For instance, we could take a map $f: C \to \mathbb P^1$ or $f: C \to A$ for $A$ an abelian variety of dimension $d$ and pull back by the following diagrams:
\[\begin{tikzcd}
C_n \arrow[d, "\pi_n"] \arrow[r, "f_n"] & \mathbb P^1 \arrow[d, "t \to t^{\ell^n}"] & or & C_n \arrow[d, "\pi_n"] \arrow[r, "f_n"] & A \arrow[d, "{[\ell^n]}"] \\
C \arrow[r, "f"]                        & \mathbb P^1                               &    & C \arrow[r, "f"]                        & A                        
\end{tikzcd}\]
We denote the first family of examples by Case A and the second family by Case B. Note that in both the families, the $C_n\to C$ are geometrically (branched) Galois extensions with an abelian Galois group $G_n \cong (\mathbb Z/\ell^n\mathbb Z)^b$ for $b=1$ or $2d$ in Case A and Case B respectively. Note that the $G_n$ themselves have an action of $\sigma_q$ and this will be crucial.

We define $M_n = H^1_{\text{\'et}}(\overline{C}_n,\mathbb Z_\ell)/H^1_{\text{\'et}}(\overline{C},\mathbb Z_\ell)$, $f_n(x)$ to be the characteristic polynomial of $\sigma_q$ on $M_n$ and $g_n$ to be the characteristic polynomial $\det\left(1 - \sigma_qx|M_n/M_{n-1}\right)$. It does not seem to be true that $g_n$ determines $g_{n+1}$ as in Example \ref{eg: motivating example}. Nonetheless, the following weaker convergence statement is true.

Let $k_n$ be the order of $\sigma_q$ acting on $\mu_{\ell^n} = \mathbb G_m[\ell^n]$ in Case A while in Case B, $k_n$ is a close relative of the order of $\sigma_q$ acting on $A[\ell^n]$. In particular, it is independent of $C$ and can be made completely explicit. In either case $k_n$ is of the form $\max\{1,\ell^{n-n_0}\}$ with $n_0$ depending on which case we are considering.

\begin{theorem*}[Theorem \ref{thm: main thm, geom}]
In the above set up (with some mild assumptions on $f$ and $q$),
\begin{enumerate}
    \item We have a factorization into monic polynomials:
            \[f_m(x) = \prod_{n\leq m}g_n(x)\]
        where the $g_n$ are independent of $m$.
    \item There exist polynomials $h_n(y), \tilde{h}_n(y) \in \mathbb Z[y]$ such that, \textbf{in Case A}
        \[g_n(x) = h_n(x^{k_n}).\]
    while in \textbf{Case B}
    \[    g_n(x) = \tilde{h}_n(x^{k_n}).
    \]
    \item 
    \textbf{In Case A:} For $n$ sufficiently large so that $k_{n+1} = \ell k_n$ (Lemma \ref{lem: order of Q}), we have the $\ell$-adic convergence
     \[  h_{n+1}(y) \equiv h_n(y) \pmod{\md}.
    \]
    In particular, the following $\ell$-adic limit exists in $\mathbb Z_\ell[y]$:
    \[h_\infty(y) = \lim_{n\to\infty}h_n(y).\]
    \textbf{In Case B:} For $n\geq n_0$ sufficiently large so that $k_{n+1} = \ell k_{n}$, we have the congruence
    \[    \tilde{h}_{n+1}(y) \equiv \tilde{h}_{n}^{\ell^{(b-1)}}(y) \pmod{\md}.
    \]
    In particular, the following $\ell$-adic limit exists in $\mathbb Z_\ell[y]$ with $\exp,\log$ defined formally as power series:
    \[\tilde{h}_\infty(y) = \exp\left(\lim_{n\to\infty}\frac{1}{\ell^{(n-n_0)(b-1)}}\log(\tilde{h}_{n}(y))\right).\]
\end{enumerate}
\end{theorem*}

The first two properties of the theorem are fairly standard \footnote{As a reviewer pointed out, Part 2 has been "known for a long time and rediscovered several times", for instance see \cite[Lemma 1.1]{gordon1979linking}. For completeness, we give our own proof too.} and follow from understanding the structure of $M_n$ as a module over $\mathbb Z_\ell[G_n,\sigma_q]$ and in particular, depends on $\sigma_q$ having "large" orbits when acting upon the characters of $G_n$. The main body of the paper proves a more abstract statement (Theorem \ref{thm: Main theorem, abstract, general}) about the convergence of certain invariants of a non abelian cohomology group which implies the third part of the above theorem on the towers of curves. 

We note that this more abstract statement can be applied to many more geometric contexts than just our two examples of towers of curves above although we do not pursue this in our paper. It applies to any tower of varieties with an action of an abelian group such that the Frobenius action on the cohomology has a "large" orbit. For instance, we could take hypersurfaces of the form
\[f(x_0^{\ell^n},\dots,x_n^{\ell^n}) = 0 \subset \mathbb P^{n}_{\mathbb F_q}.\]
All the interesting cohomology is concentrated in the middle dimensional cohomology and the above theorem holds for the characteristic polynomial of the Frobenius action on this middle dimensional cohomology group.

We will see in Section \ref{sec: Iwasawa theory} that $M_\infty = \varprojlim_n M_n$ is a free module for a certain skew-abelian Iwasawa algebra and in particular, the characteristic polynomials we study are all determined by a Galois cohomology class with coefficients in matrices over a ring of power series. The bulk of this paper consists in studying the $\ell$-adic properties of these power series.

A key role in the study of the study of these algebraic properties is played by the following generalization of Fermat's little theorem (conjectured by Arnold in \cite{Arnold} and proven by Zarelua in \cite{Zarelua}):

\begin{theorem*}[Arnold - Zarelua, Theorem \ref{thm: Arnold-Zarelua}]
Let $A$ be a $r\times r$ matrix over $\mathbb Z_\ell$. Then, the congruence
\[\tr(A^{\ell^{n+1}}) \equiv \tr(A^{\ell^n}) \pmod{\ell^{n+1}}\]
holds for any prime $\ell$ and any $n \in \mathbb N$.
\end{theorem*}

Arnold's conjecture goes back to before Arnold (J\"anichen \cite{janichen1921uber} and Schur \cite{schur1937arithmetische}). For a more recent expository survey and applications to topology and dynamics, see Zarelua \cite{Zarelua}. Arnold's conjecture has since been proven many times in the literature (for instance, see \cite{mazur2010generalizations}). We give a new proof\footnote{In the course of writing this paper, we found essentially the same proof by \href{https://rjlipton.wpcomstaging.com/2009/08/07/fermats-little-theorem-for-matrices/}{Qiaochu Yuan} in a blog post from 2009.} of a slightly refined statement since we will use a similar technique in proving our main theorem.

To keep notation simple, we state a special (yet non-trivial) case of our general $\ell$-adic convergence theorem.

\begin{theorem*}[Theorems \ref{thm: Main theorem, abstract, scalar},\ref{thm: Main theorem, abstract, general}]
Let $F(t)$ be a $r\times r$ matrix with entries in $\mathbb Z_\ell[t]$. Suppose that $q$ is a prime such that $q-1$ is divisible by $\ell$ but not $\ell^2$. For each $n\geq 1$, we define the matrix
\[A_n = \prod_{i=1}^{\ell^{n-1}}F(\zeta_{\ell^n}^{q^i})\]
with characteristic polynomial $p_n(x)$. Then, the limit $p_\infty(x) = \lim_n p_n(x)$ exists and we have the congruence
\[p_{n+1}(x) \equiv p_n(x) \pmod{\ell^n}.\]
\end{theorem*}

We note that even in the simplest case where $r=1$, the above theorem is not obvious. 

\subsection{Some questions for future work}

We pose a few questions suggested by this work. 

\begin{question}
Our main theorem establishes the existence of $\ell$-adic limits $h_\infty(x), \tilde{h}_\infty(x) $ in the two cases. In some simple cases, the $h_n(x)$ are independent of $n$ for $n$ large enough and by the proof of the Weil conjectures, are known to in fact be polynomials over $\mathbb Z$ while a-priori $h_\infty(x)$ is only defined over $\mathbb Z_\ell$. 

Are the roots of $h_\infty(x)$ always transcendental numbers (except in the cases where $h_n$ is eventually constant)?
\end{question}

\begin{question}
Even if the roots of $h_\infty(x)$ are transcendental, is it possible to describe them using simple $\ell$-adic transcendental functions? 
\end{question}

\begin{question}
What information about the original morphism $f: C \to \mathbb P^1$ does $h_\infty(x)$ remember? In the classical Iwasawa theory set up, the limiting characteristic polynomials turn out to be equal to various $\ell$-adic L functions up to a unit (by the Main Conjecture of Iwasawa Theory), can we hope for something similar in this case?
\end{question}

\begin{question}
Let $(\Lambda,\sigma_q)$ be as in Section \ref{sec: Iwasawa theory} and $M$ a finite, free $\Lambda$ module with a $\sigma_q$ semi-linear endomorphism $\Phi: M \to M$. It might be possible and interesting to completely classify such endomorphisms $\Phi$ in the hope of a more conceptual proof of the main results. This question is reminiscent of Manin's classification of Dieudonne modules (\cite{manin1963theory}). Indeed, the $(M\otimes_\Lambda\Lambda_n(v),\Phi)$ form a "compatible" system of an "$\ell$-adic analogue of Dieudonne modules" over the "compatible" system of cyclotomic local rings with endomorphism $(\Lambda_n(v),\sigma_q)$ as $n$ varies - this final sentence is purely impressionistic!
\end{question}

\begin{question}
Let $Q:\mathbb Z_\ell^n \to \mathbb Z_\ell^n$ be a linear automorphism and for $v \in \mathbb Z_\ell^n$, let $k_n(v)$ be the smallest positive integer so that $Q^{k_n(v)}v \equiv v \pmod{\ell^n}$. Let $\lambda: \mathbb Z_\ell^n \to \mathbb Z_\ell$ be an arbitrary linear form. Does the sequence 
\[S_n(\lambda,v) \coloneqq \sum_{j=1}^{k_n(v)}\zeta_{\ell^n}^{\lambda(Q^{-j}v)}\] defined in Remark \ref{rmk: failure in general case, sum of rou} converge to $0$ as $n \to \infty $? If so, what is the rate of convergence and is it uniform as $v$ ranges over primitive vectors?
\end{question}

\textit{Acknowledgements:} I would like to thank my advisor Jordan Ellenberg for posing a question that led to this paper, feedback on the writing of this paper and many other useful discussions, Douglas Ulmer for many helpful discussions and useful feedback on the writing of the paper, John Yin for helping with some computer calculations.

I am also very grateful to the anonymous referee for helpful expository suggestions and spotting an error in an earlier version of the proof of Lemma \ref{lem: Invariant cohomology} and to Yifan Wei for helping me fix the error. 

\textit{Outline of the paper:} For expository reasons, the paper is not presented in strictly logical order. Section \ref{sec: Iwasawa theory} is independent of the rest of the paper and its main results (Theorems \ref{thm: Main theorem, abstract, scalar} and \ref{thm: Main theorem, abstract, general}) are used in proving our main geometric result (Theorem \ref{thm: main thm, geom}). The reader interested in the geometry and willing to take the $\ell$-adic analysis on faith can skip Section \ref{sec: Iwasawa theory}. The reader interested only in the $\ell$-adic convergence results can skip Section \ref{sec: dist of frob ev in towers of curves}.

\section{On the cohomology of a tower of curves}\label{sec: dist of frob ev in towers of curves}

In this section, we reduce Theorem \ref{thm: main thm, geom} to an abstract statement about the convergence of characteristic polynomials of a sequence of matrices.

We fix an odd prime $\ell$ and a finite field $\mathbb F_q$ with $q$ large enough to be specified soon. For a variety $X/\mathbb F_q$, the notation $H^i_{\text{\'et}}(\overline{X},\mathbb Z_\ell)$ denotes as usual the \'etale cohomology of the variety $X\times_{\mathbb F_q}\overline{\mathbb F}_q$ with $\mathbb Z_\ell$ coefficients. The Frobenius $\sigma_q$ acts on it through a linear automorphism.

\subsection{Two families of Iwasawa towers}

\begin{definition}\label{defn: cases}

Let $C/\mathbb F_q$ be a curve. We will be interested in the following two classes of towers:

\begin{itemize}
    \item \textit{Case A:} Given a non-constant map $f: C \to \mathbb P^1$, we can construct extensions $\pi_n: C_n \to C$ by the pull back diagram
        \[\begin{tikzcd}
        C_n \arrow[d, "\pi_n"] \arrow[r, "f_n"] & \mathbb P^1 \arrow[d, "t \to t^{\ell^n}"] \\
        C \arrow[r, "f"]                        & \mathbb P^1                              
        \end{tikzcd}\]
    The $C_n$ form an inverse system with an action by the group
    \[\Gamma_n = \mathbb Z/\ell^n\mathbb Z \rtimes \mathbb Z\]
    where we denote a generator for the first factor by $\theta$ (corresponding to $\theta(t) = \zeta_{\ell^n}t$) and a generator for the second factor by $\sigma_q$ corresponding to the Frobenius operation. They satisfy the commutation identity
    \[\sigma_q\theta = \theta^q\sigma_q.\]
    We require the $\pi_n: C_n \to C$ to be totally ramified over the preimage $f^{-1}\left(\{0,\infty\}\right)$ - for instance, this is satisfied if $f$ is unramified over $0,\infty$ or more generally, if the ramification indices of $f$ over $0,\infty$ are co-prime to $\ell$. This guarantees that the $C_n$ are geometrically irreducible.

    \item \textit{Case B:} Given an Abelian variety $A/\mathbb F_q$ of dimension $d$ and a map $f: C \to A$, we construct $\pi_n: C_n \to C$ by the pullback diagram
    \[
        \begin{tikzcd}
        C_n \arrow[d, "\pi_n"] \arrow[r, "f_n"] & A \arrow[d, "{[\ell^n]}"] \\
        C \arrow[r, "f"]                        & A                        
        \end{tikzcd}
    \]
    We require the $C_n$ to be geometrically irreducible, this is achieved for instance if the induced map $\pi^{\text{\'et}}_1(f):\pi^{\text{\'et}}_1(C)\to\pi^{\text{\'et}}_1(A)$ on the \'etale fundamental groups is surjective. In this case, the $C_n$ are acted upon by (with $b = 2d$)
    \[\Gamma_n = (\mathbb Z/\ell^n\mathbb Z)^{b} \rtimes \mathbb Z.\]
    The first factor can be identified
    with $A[\ell^n]$ and we denote a basis of it by $\alpha_1,\dots,\alpha_d ,\beta_1,\dots,\beta_d$ so that the Frobenius $\sigma_q$ (corresponding to the second factor) acts by a $b\times b$ matrix $Q$ as
    \[\sigma_q v = Qv\]
    for $v \in A[\ell^n].$ The congruence $Q\equiv \Id \pmod{\ell}$ is equivalent to $\sigma_q$ acting as the identity on $A[\ell]$. This can always be achieved by a finite extension of the base field $\mathbb F_q$ and we suppose that $q$ is large enough so that $Q \equiv \Id \pmod{\ell}$.\footnote{If $\ell = 2$, we would need $Q \equiv \Id \pmod{\ell^2}$.} Note that $1$ is not an eigenvalue of $Q$ since $\sigma_q-1: A \to A$ has finite degree equal to $A(\mathbb F_q)$. 
    
    % Finally, $\sigma_q$ is known to act semi-simply on the (rational) \'etale cohomology of any abelian variety\footnote{\href{https://mathoverflow.net/a/104105/58001}{Semi-simplicity for abelian varieties}} - and conjectured to act semi-simply on any variety over a finite field whatsoever\footnote{\href{https://mathoverflow.net/a/104635/58001}{Conjecture for arbitrary varieties}}.

\end{itemize}

\end{definition}

\begin{remark}
These aren't the only cases our main theorem applies to and in fact, we can even generalize to higher dimensions. What is important is that our tower of varieties has an action by a pro-$\ell$ abelian group as above and that the growth in cohomology is "regular" in the tower so that as a module over the group algebra, the rank of the cohomology groups are constant. For example, we could take Fermat hypersurfaces of the form 
\[X_n : \sum_{j=0}^{d} x_j^{\ell^n} = 0 \subset \mathbb P^{b+1}\]
with action by $G_n = (\mathbb Z/\ell^n\mathbb Z)^b$. The only interesting cohomology group is in degree $i = b$, in which case it is a rank $1$ module over $\mathbb Z_\ell[G_n]$ (\cite[Theorem 6]{anderson1987torsion} for instance) and a straightforward variant of Theorem \ref{thm: structure of reps} shows that growth in cohomology is regular.

\end{remark}

\begin{remark}
Note that the automorphism groups $\Gamma_n$ aren't abelian but they are very close to being abelian, being the extension of an abelian group by the Frobenius action. Therefore one could view this as an example of skew-abelian Iwasawa theory.

\end{remark}

In the remainder of this subsection we prove some basic results about the cohomology of these towers (ignoring the Frobenius action initially).

\begin{lemma}\label{lem: direct factor}
For a finite extension of curves $f: X \to Y$, $H^1_{\text{\'et}}(\overline{Y},\mathbb Z_\ell)$ is a direct factor of $H^1_{\text{\'et}}(\overline{X},\mathbb Z_\ell)$.
\end{lemma}
\begin{proof}
Since $H^1_{\text{\'et}}(\overline Y,\mathbb Z_\ell)$ is dual to the Tate module $T_\ell(Y)$, it suffices to show the corresponding fact for the Tate modules of $X$ and $Y$, i.e., we need to show that the natural map $f: T_\ell(X) \to T_\ell(Y)$ is surjective and that the kernel is torsion free.

One easily checks the following composite map
\[\mathrm{Jac}(Y) \xrightarrow{f^*} \mathrm{Jac}(X) \xrightarrow{f_*} \mathrm{Jac}(Y)\]
is simply multiplication by the degree of $f$, for instance by using an isomorphism $\mathrm{Jac}(X) \cong \mathrm{Pic}(X)$ and computing the map explicitly in terms of divisors supported away from the ramification locus. This shows that the second map is surjective which in turn implies that the map on Tate modules $T_\ell(f): T_\ell(X) \to T_\ell(Y)$ is surjective.

Moreover, the kernel of $T_\ell(f)$ is torsion free since if $[P_n]_{n\geq 1} \in T_\ell(X)$ mapped to zero, then $P_n \in \mathrm{ker}(f)$ which would imply that $\deg(f) \geq \ell^n$ for all $n$ which is a contradiction.
\end{proof}

When the extension is generically Galois, we can say more.

\begin{lemma}\label{lem: Invariant cohomology}
Suppose $f: X \to Y$ is a generically Galois (branched) extension of (smooth, proper) curves with Galois group $G$. Then, $H^1_{\text{\'et}}(\overline Y,\mathbb Z_\ell)$ is exactly the submodule of $H^1_{\text{\'et}}(\overline X,\mathbb Z_\ell)$ fixed by the $G$ action.

\end{lemma}
\begin{proof}
Let us first suppose that $X,Y$ are not necessarily proper but that $f:X\to Y$ is unramified. By the Hochschild-Serre spectral sequence, \[H^r(G,H^s_{\text{\'et}}(\overline X,\mathbb Z_\ell)) \Longrightarrow H^{r+s}_{\text{\'et}}(\overline Y,\mathbb Z_\ell).\]
If we want to let $r+s = 1$, then we have either $r=0,s=1$ or $r=1,s=0$. But, $H^0_{\text{\'et}}(\overline X,\mathbb Z_\ell) = \mathbb Z_\ell$ with the trivial $G$ action and therefore
\[H^1(G,H^0_{\text{\'et}}(\overline X,\mathbb Z_\ell)) = \mathrm{Hom}(G,\mathbb Z_\ell) = 0\]
since $G$ is torsion and $\mathbb Z_\ell$ is torsion free. This causes the spectral sequence to degenerate at the $(1,1)$ term and we have the required isomorphism
\[H^1_{\text{\'et}}(\overline Y,\mathbb Z_\ell) \cong H^0(G,H^1_{\text{\'et}}(\overline X,\mathbb Z_\ell)).\]
Now, for a general (branched) $f: X \to Y$, let $T \subset Y$ be the ramification divisor on $Y$ and $f^{-1}(T) = S\subset X$ its preimage in $X$ with $U = X-S,V = Y-T$. With this set-up, we have following commutative diagram
\[\begin{tikzcd}
	{H^1_{\text{\'et}}(\overline X,\mathbb Z_\ell)} & {H^s_{\text{\'et}}(\overline U,\mathbb Z_\ell)} \\
	{H^1_{\text{\'et}}(\overline Y,\mathbb Z_\ell)} & {H^1_{\text{\'et}}(\overline V,\mathbb Z_\ell)}
	\arrow[hook, from=1-1, to=1-2]
	\arrow[hook, from=2-1, to=1-1]
	\arrow[hook, from=2-2, to=1-2]
	\arrow[hook, from=2-1, to=2-2].
\end{tikzcd}\]

Note that the cokernels along the horizontal rows have weight $2$ (i.e., $\sigma_q$ acts by $q$ on the cokernel) as can be seen either from the excision long exact sequence or from the Lefschetz fixed point theorem for compactly supported cohomology along with Poincar\'e duality. On the other hand $H^1_{\text{\'et}}(\overline Y,\mathbb Z_\ell),H^1_{\text{\'et}}(\overline Y,\mathbb Z_\ell)$ are both of weight $1$ (by the Weil conjectures, for instance).

The above diagram is $G$-equivariant since $S,T$ are. Therefore, the $G$-invariants of $H^1_{\text{\'et}}(\overline X,\mathbb Z_\ell)$ are contained in $H^1_{\text{\'et}}(\overline V,\mathbb Z_\ell)$ but the above weight argument shows that it is in fact contained in $H^1_{\text{\'et}}(\overline Y,\mathbb Z_\ell)$ as required.
\end{proof}

Let us return to our specific towers above. 

\begin{definition}\label{defn: Gn}
In \textit{Case A}, let $G_n = \mathbb Z/\ell^n\mathbb Z$ with generator $\theta$ while in \textit{Case B}, let $G_n = (\mathbb Z/\ell^n\mathbb Z)^{b}$ with generators $\alpha_i,\beta_j$ as discussed before. We also define the group algebra $R_n = \mathbb Z_\ell[G_n]$.

\end{definition}

By Lemma \ref{lem: direct factor} and \ref{lem: Invariant cohomology}, $M_n = H^1_{\text{\'et}}(\overline{C}_n,\mathbb Z_\ell)/H^1_{\text{\'et}}(\overline{C},\mathbb Z_\ell)$ is a free $\mathbb Z_\ell$ module with an action of $R_n$ described by the following theorem with $g_0$ the genus of $C$.

\begin{theorem}\label{thm: structure of reps}

Let us define
\[r = \begin{cases}2g_0 + s - 2 &\text{ in \textit{Case A}}\\ 2g_0-2 &\text{ in \textit{Case B}}\end{cases}\]
where in \textit{Case A}, $s$ is the number of preimages of $0,\infty$ for the defining map $f: C \to \mathbb P^1$.

As $R_n$ modules, we have an exact sequence
\[0 \to \mathbb Z_\ell^r \to R_n^r \to M_n \to 0.\]
where $G_n$ acts trivially on the first term.
\end{theorem}

As a preliminary to the above theorem, we use Riemann-Hurwitz to compute the dimensions of $M_n$.

\begin{lemma}
Let $g_n$ be the genus of $C_n$. $M_n$ is a free $\mathbb Z_\ell$ module of rank $2(g_n-g_0)$ and in \textit{Case A}, we have
\[\dim_{\mathbb Z_\ell} M_n = (\ell^n-1)(2g_0+s-2) \]
while in \textit{Case B}, we have
\[\dim_{\mathbb Z_\ell} M_n = (\ell^{bn}-1)(2g_0-2)\]
\end{lemma}
\begin{proof}
By Lemma \ref{lem: direct factor} and \ref{lem: Invariant cohomology}, $M_n$ is a free $\mathbb Z_\ell$ module. It remains to compute its $\mathbb Z_\ell$ rank ($= 2(g_n-g_0)$). In \textit{Case A}, let$S_0,S_\infty \subset C(\overline{\mathbb F}_q)$ be the preimages of $0,\infty$ under $f$ so that $s = |S_0|+|S_\infty|$. Note that $\pi_n$ is only ramified over $S_0,S_\infty$ and by assumption, it is totally ramified to order $\ell^n$ over these points. By Riemann-Hurwitz, we then have
\[2g_n - 2 = \ell^n(2g_0-2) + s(\ell^n-1) \implies 2(g_n-g_0) = (\ell^n-1)(2g_0 + s-2).\]

In \textit{Case B}, $\pi_n$ is unramified and of degree $\ell^{bn}$ and therefore, we simply have
\[2g_n - 2 = \ell^{bn}(2g_0-2) \implies 2(g_n-g_0) = (\ell^{bn}-1)(2g_0-2).\]
\end{proof}

We finish the proof of Theorem \ref{thm: structure of reps} by using the Lefschetz fixed point theorem to compute the character of $M$ in terms of fixed points.

\begin{proof}

In \textit{Case A}, let $g\in G_n$ be non trivial. Since $g$ is not the identity, the only points it fixes on $C_n$ are the points lying over $0,\infty$ under the map $C_n \to \mathbb P^1$. In the notation of the previous lemma, there are $s$ such points in total and the local index at each point is $+1$. Moreover, $g$ acts trivially on the degree $0,2$ cohomology groups. Therefore, by the Lefschetz fixed point formula
\[\tr(g|H^1_{\text{\'et}}(\overline{C}_n,\mathbb Z_\ell)) = 2 - s\]
and since $G$ acts trivially on $C_0$, 
\[\tr(g|H^1_{\text{\'et}}(\overline{C}_n,\mathbb Z_\ell)) - \tr(g|H^1_{\text{\'et}}(\overline{C},\mathbb Z_\ell)) = -(2g_0+s-2) = -r.\]
On the other hand, the identity $\mathrm{id} \in G_n$ of course acts trivially so that
\[\tr(\mathrm{id}|H^1_{\text{\'et}}(\overline{C}_n,\mathbb Z_\ell)) - \tr(\mathrm{id}|H^1_{\text{\'et}}(\overline{C},\mathbb Z_\ell)) = 2(g_n-g_0) = r(\ell^n-1)\]
where the final equality is by the previous lemma.

In \textit{Case B}, any $g \neq \mathrm{id} \in G_n$ acts on the abelian variety $A$ by a non trivial translation and hence has no fixed points on either $C_n$ or $A$. As before, by the Lefschetz fixed point theorem
\[\tr(g|H^1_{\text{\'et}}(\overline{C}_n,\mathbb Z_\ell)) - \tr(g|H^1_{\text{\'et}}(\overline{C},\mathbb Z_\ell)) = 2-2g_0 = -r.\]
The identity element has trace equal to
\[\dim H^1_{\text{\'et}}(\overline{C}_n,\mathbb Z_\ell) - \dim H^1_{\text{\'et}}(\overline{C},\mathbb Z_\ell) = 2(g_n-g_0) = r(\ell^{bn}-1).\]
If we then examine the exact sequence
\[0 \to \mathbb Z_\ell^r \to R_n^r \to X \to 0,\]
we see that $X$ has the character we computed above in both cases for $M_n$ proving that $X \cong M_n$ as $G_n$ representations.

\end{proof}

\begin{remark}\label{rmk: characters of Mn}
As an immediate corollary of the above theorem, we notice that in both \textit{Case A} and \textit{B}, for every non trivial character $\chi: G_n \to \overline{\mathbb Z}^\times$, the corresponding eigenspace $M_{n}(\chi)$ of $M_n\otimes \mathbb Z_\ell[\zeta_{\ell^n}]$ is of dimension $2g_0+s-2$ and $2g_0-2$ in the two cases respectively. In particular it is independent of $n$ and we call the characters appearing in $P_n = M_n/M_{n-1}$ "new" or "primitive" characters of level $n$.

We fix a set of generators $t_1,\dots,t_b$ for $G_n \cong (\mathbb Z/\ell^n\mathbb Z)^b$ and identify characters $\chi$ of $G_n$ by vectors $v = (v_1,\dots,v_b) \in (\mathbb Z/\ell^n\mathbb Z)^b$ by defining \(\chi_v(t_i) = t_i^{v_i}.\) Under this identification, primitive characters correspond exactly to primitive vectors as defined below in Definition \ref{defn: primitive vectors}. We denote the eigenspace of $\chi_v$ by $M_n(v)$.

\end{remark}

The exact sequence in the above theorem implies that $M_n$ is not a free $R_n$ module but nevertheless, the inverse limit $M_\infty \coloneqq \lim_n M_n$ is a free module over $\Lambda = \mathbb Z_\ell[[T_1,\dots,T_b]] = \lim_n R_n$.

\begin{lemma}\label{lem: limiting M_n}
Let $\theta_1,\dots,\theta_b$ be the generators of $G_n = \left(\mathbb Z/\ell^n\mathbb Z\right)^b$ as above. Then the projective limit $M_\infty \coloneqq \varprojlim_n M_n$ is a free module of rank $r$ over $\Lambda = \mathbb Z_\ell[\![T_1,\dots,T_b]\!]$. The Frobenius $\sigma_q$ acts \emph{semi}-linearly on $M_\infty$, i.e., $\sigma_q$ is $\mathbb Z_\ell$ linear and satisfies
\[\sigma_q\circ(1+T_i) = {\sigma_q}(1+T_i)\circ\sigma_q\]
where we identify $1+T_i$ with $\theta_i$ so that $\sigma_q$ acts on $1+T_i$ through its action on $\varprojlim_n G_n$.
\end{lemma}
\begin{proof}
By the above theorem, we have the following identification as $\mathbb Z_\ell[G_n]$-modules
\[M_n \cong \left(\frac{\mathbb Z_\ell[\theta_1,\dots,\theta_b]}{(\theta_1^{\ell^n} = 1,\dots,\theta_b^{\ell^n} = 1, \prod_{i=1}^b\left(\sum_{j=0}^{\ell^n-1}\theta_i^{j}\right)}\right)^r\]
since $\prod_{i=1}^b\left(\sum_{j=0}^{\ell^n-1}\theta_i^{j}\right)$ generates the unique $1$-dimensional $\mathbb Z_\ell$ submodule of $\mathbb Z_\ell[G_n]$ with trivial $G_n$ action. Using this explicit presentation, we define a map
\[\Lambda^r = \left(\mathbb Z_\ell[\![T_1,\dots,T_b]\!]\right)^r \to M_\infty\]
by mapping, for each factor, the $T_i \to \theta_i-1$ in each term in the projective limit. We will prove that this map is an isomorphism. Since the map is defined on each factor, we can assume henceforth that $r=1$. The kernel of the induced map to $M_n$ is generated by the elements
\[(1+T_i)^{\ell^n}-1 = \sum_{j=1}^{\ell^n}\binom{\ell^n}{j}T_i^j \text{   for    } i=1,\dots, b\]
and
\[\prod_{i=1}^b\left(\frac{(1+T_i)^{\ell^n}-1}{T_i}\right) = \prod_{i=1}^b\left(\sum_{j=1}^{\ell^n}\binom{\ell^n}{j}T_i^j\right).\]
As $n \to \infty$, these elements tend to $0$ in the $(\ell,T_1,\dots,T_b)$-adic topology of $\Lambda$ so that the map $\Lambda \to M_\infty$ is injective. On the other hand, surjectivity is also clear since the $\theta_i$ generate $G_n$, and consequently the $\theta_i-1$ generate $\mathbb Z[G_n]$. The Frobenius action is induced through this morphism, thus completing the proof.
\end{proof}

\subsection{On the distribution of Frobenius eigenvalues in towers of curves}

In this subsection, we prove that the characteristic polynomials $f_n(x)$ of $\sigma_q$ on $H^1_{\text{\'et}}(\overline{C}_n,\mathbb Z_\ell)$ in our two cases satisfy some striking congruences. We will treat the cases uniformly by letting $Q =q, b= 1$ in Case A.

\begin{definition}\label{defn: primitive vectors}
For $R$ a discretely valued ring (DVR) or a quotient of a DVR, we call $v \in R^b$ primitive if at least one of its co-ordinates is a unit. We denote the space of primitive vectors by $\PP(R^b)$.
\end{definition}

For a primitive vector $v \in H^1_{\text{\'et}}(\overline{C}_n,\mathbb Z_\ell)$, we define $k_n(v)$ to be the smallest postive integer such that $Q^{k_n(v)}v \equiv v \pmod{\ell^n}$. We define $k_n$ to be the minimum of $k_n(v)$ as $v$ ranges over primitive vectors. Lemma \ref{lem: order of Q} shows the existence of a positive integer $\beta_v$ such that $k_n(v) = \ell^{n-\beta_v}$ for $n \geq \beta_v$. Moreover, $n_0 = \max_{v \text{ primitive}}\beta_v$ is finite so that $k_n = \ell^{n-n_0}$ for $n\geq n_0$.

\begin{theorem}\label{thm: main thm, geom}
Let $C_n$ be as in Case A or B of Definition \ref{defn: cases} and 
\[f_n(x) = \det\left(1 - \sigma_qx|M_n\right)\]
be the characteristic polynomial of the Frobenius $\sigma_q$ acting on $M_n = H^1_{\text{\'et}}(\overline{C}_n,\mathbb Z_\ell)/H^1_{\text{\'et}}(\overline{C},\mathbb Z_\ell)$. It satisfies the following properties:
\begin{enumerate}
    \item We have a factorization into monic polynomials:
    \begin{equation}\label{eqn: main thm 1, Case A}
        f_m(x) = \prod_{n\leq m}g_n(x)
    \end{equation} 
    where the $g_n$ are independent of $m$.
    \item There exist polynomials $h_n(y), \tilde{h}_n(y)$ such that, \textbf{in Case A}
    \begin{equation}\label{eqn: main thm 2, Case A}
        g_n(x) = h_n(x^{k_n}).
    \end{equation}
    while \textbf{in Case B}
    \begin{equation}\label{eqn: main thm 2, Case B}
        g_n(x) = \tilde{h}_n(x^{k_n}).
    \end{equation}
    \item 
    \textbf{In Case A:} For $n \geq n_0$ (Lemma \ref{lem: order of Q}), we have the $\ell$-adic convergence
    \begin{equation}\label{eqn: main thm 3, Case A}
        h_{n+1}(y) \equiv h_n(y) \pmod{\md}.
    \end{equation}
    In particular, the following $\ell$-adic limit exists in $\mathbb Z_\ell[y]$:
    \[h_\infty(y) = \lim_{n\to\infty}h_n(y).\]
    \textbf{In Case B:} For $n \geq n_0$, we have the congruence
    \begin{equation}\label{eqn: main thm 3, Case B}
        \tilde{h}_{n+1}(y) \equiv \left(\tilde{h}_n(y)\right)^{\ell^{(b-1)}} \pmod{\md}.
    \end{equation}
    In particular, the following $\ell$-adic limit exists in $\mathbb Z_\ell[y]$ with $\exp,\log$ defined formally as power series:
    \[\tilde{h}_\infty(y) = \exp\left(\lim_{n\to\infty}\frac{1}{\ell^{(n-n_0)(b-1)}}\log(\tilde{h}_{n}(y))\right).\]
\end{enumerate}
\end{theorem}

\begin{remark}
The Frobenius $\sigma_q$ is known to act semi-simply on the \'etale cohomology of a curve\footnote{\href{https://mathoverflow.net/a/104105/58001}{Semi-simplicity for abelian varieties}} and conjectured to act semi-simply with rational coefficients on any variety over $\mathbb F_q$. While the following proof simplifies slightly if we use the semi-simplicity of $\sigma_q$ on $M_n$, we do not assume this so that the following proof can be adapted more easily to cases where semi-simplicity is not known.
\end{remark}

\begin{proof}
\textbf{Part $1$}, i.e., Equation \ref{eqn: main thm 1, Case A} is an immediate consequence of Lemma \ref{lem: direct factor} once we define $g_n(x)$ to be the characteristic polynomial of $\sigma_q$ on $P_n = M_n/M_{n-1}$.

\textbf{To prove Part $2$}, i.e., Equations \ref{eqn: main thm 2, Case A} and \ref{eqn: main thm 2, Case B}, we treat the two cases simultaneously by taking $b=1,Q = q$ in Case A. Recall the notation that, for $v \in \mathbb Z_\ell^b$, $M_n(v)$ is the eigenspace of $G_n$ for the character $\chi_v(t_i) = t_i^{v_i}$. The eigenspaces $M_n(v)$ get permuted by $\sigma_q$ in the following manner:
\[\sigma_q: M_n(v) \to M_n(Q^{-1}v)\]
and therefore $\sigma_q^{k_n(v)}$ is an automorphism of $M_n(v)$. We will prove that a Jordan block of $\sigma_q^{k_n(v)}$ acting on $M_n(v) \subset P_n$ (with eigenvalue $\lambda \neq 0$) corresponds to $k_n(v)$ distinct Jordan blocks of $\sigma_q$ acting on $P_n$ (with eigenvalues $\mu^{1/k_n(v)}$). Since this claim is independent of passing to an extension, we replace $P_n$ by $P_n\otimes_{\mathbb Z_\ell}\overline{Q}_\ell$.

To that end, let $m_1,\dots,m_s \in M_n(v)$ be some generalized eigenvectors of $\sigma_q^{k_n}$ corresponding to a pure Jordan block of eigenvalue $\lambda$ (possibly defined over an extension $\mathbb Z_\ell$) so that
\[\sigma_q^{k_n}(m_{i+1}) = \lambda m_{i+1} + m_{i}\]
(with the convention that $m_0 = 0$.) We will first show that the eigenvector $m_1$ for $\sigma_q^{k_n(v)}$ corresponds to $k_n(v)$ distinct eigenvectors for $\sigma_q$. For $m_i \in M_n(v)$, let $m_{i,j} = \sigma_q^{j-1}(m_i)$ for $j=1,\dots,k_n(v)$. Note that \[\sigma_{q}(m_{i+1,k_n(v)}) = \sigma_q^{k_n(v)}(m_{i+1}) = \lambda m_{i+1,1} + m_{i,1}.\]
For each $\mu$ a $k_n$ root of $\lambda$, \(n_\mu = \sum_{j=1}^{k_n(v)}\mu^{-j}m_{1,j}\)
is an eigenvector of $\sigma_q$. Indeed, we have
\[\sigma_q(n_\mu) = \sum_{j=1}^{k_n(v)-1}\mu^{-j}m_{1,j+1} + \lambda\mu^{-k_n(v)} m_{1,1}  = \mu n_\mu .\]
Therefore, the $n_\mu$ are each an eigenvector of $\sigma_q$  and the subspace $N = \mathrm{span}(n_\mu : \mu^{k_n(v)} = \lambda)$ is stable under $\sigma_q$ and contains 
\[m_{1,j} = \frac{1}{k_n(v)}\sum_{\mu^{k_n(v)} = \lambda} \mu^jn_{\mu} \ \ \ \ \ \ \text{ for } j=1,\dots,k_n(v).\]
Passing to the quotient $P_n/N$ therefore corresponds to replacing the $m_1,\dots,m_s$ by $m_2,\dots,m_s$ (with $m_2$ now an eigenvalue of $\sigma_q^{k_n(v)}$) and we continue inductively to show that each $m_i$ corresponds to $k_n(v)$ distinct generalized eigenvectors $n_{i,\mu}$ with eigenvalue $\mu$.

Let $g_{n,v}(x) = \det(\Id - \sigma_qx)$ be the characteristic polynomial of $\sigma_q$ on $N_n(v) = \bigoplus_{i=0}^{k_n(v)-1}M_n(Q^i(v))$. This module has dimension exactly $k_n(v)$ times the dimension of $M_n(v)$ and since for each generalized eigenvector $m_i$ of $M_n(v)$, we have constructed $k_n(v)$ distinct generalized eigenvectors $n_{i,\mu}$ of $N_n(v)$ corresponding to the $k_n(v)$ distinct roots of $\lambda$, the $n_{i,\mu}$ together in fact span $N_n(v)$.

The identity
\[\prod_{j=1}^{k_n(v)}(1-x\mu\zeta_{k_n(v)}^j) = 1-x^{k_n(v)}\mu^{k_n(v)}\]
then shows that $g_{n,v}(x) = h_{n,v}(x^{k_n(v)})$ for some polynomial $h_{n,v}(y)$ with roots $y = \lambda = \mu^{k_{n}(v)}$. We note that the above proof in fact computes the $h_{n,v}(x)$ to be exactly the characteristic polynomial of $\sigma_q^{k_n(v)}$ on $M_n(v)$. Since 
\[g_n(x) = \prod_{v \in \PP(\mathbb Z/\ell^{n}\mathbb Z)/\sim} g_{n,v}(x)\]
where the product is over a set of representatives for the $\sigma_q$ action on primitive vectors, the proof of Part (2) in Case A is completed by defining 
\[h_n(y) = \prod_{v \in \PP(\mathbb Z/\ell^{n}\mathbb Z)/\sim} h_{n,v}(y)\]
and setting $y = x^{k_n}$.

For Case B, we define (again as a product over a similar set of representatives for the $\sigma_q$ action on primitive vectors)
\[\tilde{h}_n(y) = \prod_{v \in \PP(\mathbb Z/\ell^{n}\mathbb Z)^b/\sim} h_{n,v}(y^{k_n(v)/k_n})\]
so that (with $y = x^{k_n}$)
\[g_n(x) = \prod_{v \in \PP(\mathbb Z/\ell^{n}\mathbb Z)^b/\sim} g_{n,v}(x) = \prod_{v \in \PP(\mathbb Z/\ell^{n}\mathbb Z)^b/\sim}h_{n,v}(x^{k_n(v)}) = \tilde{h}_n(x^{k_n}).\]

\textbf{Finally, we prove Part (3)}, i.e., equations \ref{eqn: main thm 3, Case A} and \ref{eqn: main thm 3, Case B}. Let us fix a generating set $m_1,\dots,m_r$ for $M_n$ over $\mathbb Z_\ell[G_n] \cong \mathbb Z_\ell[t_1,\dots,t_b]/(t_i^{\ell^n}-1 : i=1,\dots,b)$. Since $M_n$ is not a free $\mathbb Z_\ell[G_n]$ module, it might not be completely clear what a generating set should mean. For our purposes, it suffices to choose $m_1,\dots,m_r$ so that under any specialization that maps the $t_i$ to $\ell^n$ roots of unity, the $m_i$ specialize to a genuine basis over the induced specialization of $M_n$. That this is indeed possible follows from the explicit description of the $M_n$ as $G_n$ modules in Lemma \ref{lem: limiting M_n}. Such a specialization corresponds to a representation $\chi_v: G_n \to \overline{\mathbb Q}$ for $v \in \mathbb Z_\ell^b$ and we denote the induced specialization also by $\chi_v: M_n \to M_n(v)$ 

In terms of the $m_1,\dots,m_r$, $\sigma_q$ acting on $M_n$ can be represented by some invertible matrix $F(t_1,\dots,t_b)$. From this point on, \textit{we will be concerned only with this matrix} $F(t_1,\dots,t_b)$. Since $\sigma_q$ skew commutes with the $t_i$, we have
\[\sigma_q^{k_n(v)} = \prod_{i=1}^{k_n(v)}F(t^{Q^{k_n(v)-i}v}).\]
Therefore, with respect to the basis $\chi_v(m_1),\dots,\chi_v(m_r)$ of $M_n(v)$, the action of $\sigma_q^{k_n(v)}$ corresponds to evaluating the above product using the character $\chi_v$ and is represented by the matrix 
\[A_n(v) =  \prod_{i=1}^{k_n(v)}F(\zeta_{\ell^n}^{Q^{k_n(v)-i}v})\]
of Section \ref{sec: Iwasawa theory} (and we note that $A_n(v)$ is independent of our choice of $F$ or the $m_1,\dots,m_r$). As noted above, the $h_{n,v}(y)$ are the characteristic polynomials of $\sigma_q^{k_n(v)}$ on $M_n(v)$ and therefore, correspond to the $p_{n,v}(y)$ in Section \ref{sec: Iwasawa theory}. We further see that the $\tilde{h}_n(y)$ correspond to the polynomials $r_n(y)$ of Theorem \ref{thm: Main theorem, abstract, general} and by this theorem, we have the required congruence:
\[\tilde{h}_{n+1}(y) \equiv \tilde{h}_n(y) \pmod{\md}.\]

\end{proof}

\section{On the convergence of a skew-abelian Iwasawa theoretic invariant}\label{sec: Iwasawa theory}

In this section, we prove a general, abstract result about the convergence of a certain cohomological invariant defined for a skew commutative Iwasawa algebra. The set up is as follows. 

We fix an odd\footnote{As usual, the arguments of this paper go through if $\ell = 2$ with minor, standard modifications.} prime $\ell$ and positive integers $b,r$ throughout this section. All cohomology groups in this section represent group cohomology unless indicated otherwise. All congruences in this paper are in $\mathbb Z_\ell$ (and hence only concerned with the $\ell$-adic valuation) unless explicitly mentioned otherwise.

Let $\Lambda = \mathbb Z_\ell[[T_1,\dots,T_b]]$ be the $b$ dimensional Iwasawa algebra and set $t_i = 1 + T_i$. It is a local ring with maximal ideal $\mathfrak m = (\ell,T_1,\dots,T_b)$. Note that for $\lambda \in \mathbb Z_\ell$, the expression
\[t_i^{\lambda} = (1+T_i)^{\lambda} =  \sum_{k\geq 0}\binom{\lambda}{k} T_i^k\]
converges in $\Lambda$.
For $v = (v_1,\dots,v_b) \in (\mathbb Z_\ell)^b$, we define \(t^v = (t_1^{v_1},\dots,t_b^{v_b}).\)
We suppose that $\Lambda$ has an endomorphism $\sigma_q$ acting through a matrix $Q = Q_{ij} \in \GL_b(\mathbb Z_\ell)$ in the following way:
\[\sigma_q(t^v) = t^{Qv} \iff \sigma_q(T_i) = \left[\prod_j(1+T_j)^{Q_{ji}}\right] - 1 \text{ for all } i.\]

We note that the action is well defined since $\sigma_q(T_i) \in \mathfrak m$. For $v \in \mathbb Z_\ell^b$, we denote the size of the orbit of $v$ under $Q$ in $(\mathbb Z_\ell/\ell^n\mathbb Z_\ell)^b$ by $k_n(v)$.\footnote{i.e., $Q^{k_n(v)}v \equiv v \pmod{\ell^n}$ and $k_n(v)$ is the least such positive integer.} We also define
\[k_n = \min_{v \text{ primitive}} k_n(v).\]

\textit{Assumption:}\label{rmk: assumptions on Q} We suppose henceforth that $Q \equiv \Id \pmod{\ell}$ and that $Q$ fixes no vectors\footnote{If $\ell=2$, then we would need to assume that $Q \equiv 1 \pmod{4}$.}. 

\begin{lemma}\label{lem: order of Q}
Let $v$ be a primitive vector. Then there exist integers $\alpha\geq 1, \beta_v \geq 0$ so that
\[k_n(v) = \begin{cases} \ell^{n-\alpha-\beta_v} & \text{ if } n \geq \alpha+\beta_v\\ 1 &\text{ otherwise.}\end{cases}\] Moreover, there is some (minimal) $\beta_0$ such that $\beta_v \leq \beta_0$ for all primitive $v$. 

In particular, we have
\[k_n = \begin{cases} \ell^{n-\alpha-\beta_0} & \text{ if } n \geq n_0 \coloneqq \alpha+\beta_0\\ 1 &\text{ otherwise.}\end{cases}\]

\end{lemma}

\begin{proof}

Since $Q \equiv \Id \pmod{\ell}$, we have $\log Q = \ell^\alpha X$ for $\alpha \geq 1$ with $X\in M_b(\mathbb Z_\ell)$ not divisible by $\ell$. Since $\ell\geq 3$,
\[(Q^m-\Id)v = \exp(m\log Q)v - v = m\ell^\alpha Xv + \frac{(m\ell^\alpha)^2}{2}X^2v + \dots .\]
Since $Q$ does not fix any vectors, $Xv \neq 0$ so let $\beta_v$ be the largest value such that $Xv \equiv 0 \pmod{\ell^{\beta_v}}$. We see that $k_n(v)$ is the smallest $m$ so that $(Q^m-\Id)v$ is divisible by $\ell^n$. Since $X^kv \equiv 0 \pmod{\ell^{\beta_v}}$ for any $k\geq 1$ too, the $\ell$-adic valuation of $(Q^m-\Id)v$ is determined by the leading term $m\ell^\alpha Xv$ so that
\[k_n(v) = \begin{cases} \ell^{n-\alpha-\beta_v} & \text{ if } n \geq \alpha+\beta_v\\ 1 &\text{ otherwise.}\end{cases}\]It remains to show that there is a uniform upper bound on $\beta_v$.

Let $\pi: \mathbb Z_\ell^b \to \mathbb F_\ell^b$ be the reduction map. The primitive vectors correspond to the subspace $\PP = \pi^{-1}(\mathbb F_\ell^{b}-\{0\})$ which is a closed (and open) subset of $\mathbb Z_\ell^b$. Therefore $\PP$ is compact and by continuity of multiplication by $X$,
\[X\PP = \{Xv : v \in \PP\} \subset \mathbb Z_\ell^b\]
is compact and closed too. By assumption on $Q$, $X\PP$ does not contain $0$ (since this would correspond to a fixed point of $Q$). This implies that $X\PP$ is in fact bounded away from $0$, i.e, there is some minimal $\beta_0$ so that the image of $X\PP$ in $(\mathbb Z/\ell^{\beta_0+1}\mathbb Z_\ell)^b$ does not contain $0$ so that $\beta_v \leq \beta_0$ for every primitive $v$ (and $\beta_0 = \beta_v$ for some primitive $v$).

\end{proof}

\begin{remark}
It is easy to see why we need to restrict to $v$ primitive and to $Q$ not having any fixed vectors. If $Qv=v$, then $k_n(v) = 1$ and if $v = \ell^sv_0$, then $k_n(v) = 1$ for $n \leq s$ which is an obstruction to a uniform bound on $n$.

\end{remark}

\subsection{A cohomological interpretation}

Let $M$ be a free $\Lambda$ module of rank $r$ with a $\Lambda$-linear endomorphism $\Phi: M \to M$. Upon picking a basis $m_1,\dots,m_r$ for $M$, we express $\Phi$ as a matrix $F(T_1,\dots,T_b)$ with entries in $\Lambda$. We suppose that $\Phi$ skew commutes with $\sigma_q$ in the following sense:
\[\sigma_q\circ F = F(\sigma_q(T_1),\dots,\sigma_q(T_b))\circ\sigma_q.\]
Note that $\sigma_q$ acts on $\GL_r(\Lambda)$ through its action on $\Lambda$. This data of $M$ and the endomorphism $\Phi$ as above gives rise to an element $\eta$ in the non abelian cohomology group $H^1 (\mathbb Z\sigma_q, \GL_r(\Lambda))$ in the following way: 

Given a $F$ as above, we can define a cocycle representative by \(\eta(\sigma_q) = F \in \GL_r(\Lambda).\) A change of basis by a matrix $P \in \GL_r(\Lambda)$ corresponds to $F \to P(\sigma_q(T_1),\dots,\sigma_q(T_b))FP^{-1}$ which is exactly the boundary action. Therefore, the cohomology class $\eta \in H^1(\mathbb Z\sigma_q, \GL_r(\Lambda))$ depends only on $(M,\Phi)$. 

For a positive integer $n$ and $v = (v_1,\dots,v_b) \in \mathbb Z_\ell^b$, note that since $T_i = \zeta_{\ell^n}^{v_i} - 1$ is in the maximal ideal of $\mathbb Z_\ell[\zeta_{\ell^n}]$, we can define the quotient
\[\Lambda_n(v) = \frac{\mathbb Z_\ell[[T_1,\dots,T_b]]}{(t_1 = \zeta_{\ell^n}^{v_1},\dots,t_b = \zeta_{\ell^n}^{v_b})}.\]
We note that $\sigma_q^{k_n(v)}$ fixes the ideal $(t_1 - \zeta_{\ell^n}^{v_1},\dots,t_b - \zeta_{\ell^n}^{v_b}) \subset \mathbb Z_\ell[[T_1,\dots,T_b]]$ and thus descends to an endomorphism of $\Lambda_n(v)$.

Henceforth, we fix $\eta \in H^1(\mathbb Z\sigma_q,\GL_r(\Lambda)), v \in \mathbb Z_\ell^b$ and define the following sequence of invariants (implicitly depending on $\eta$) taking values in polynomials in one variable:
\[p_{n,-}(y): v\in H^1(\mathbb Z\sigma_q,\GL_r(\Lambda)) \xrightarrow{\mathrm{restriction}} H^1(\mathbb Z\sigma_q^{k_n(v)},\GL_r(\Lambda_n(v)))  \xrightarrow{\text{char poly}} \Lambda_n(v)[y] \ni p_{n,v}(y)\]
where for the first map, we restrict along $\mathbb Z\sigma_q^{k_n(v)}\subset \mathbb Z\sigma_q$ and push forward along the quotient $\GL_r(\Lambda) \to \GL_r(\Lambda_n(v))$ and for the second map, since $\sigma_q^{k_n(v)}$ acts trivially on $\GL_r(\Lambda_n(v))$, we have
\[H^1(\mathbb Z\sigma_q^{k_n(v)},\GL_r(\Lambda_n(v))) = \mathrm{Hom}(\mathbb Z\sigma_q^{k_n(v)}, \GL_r(\Lambda_n(v)))/\text{conjugacy} = \GL_r(\Lambda_n(v))/\text{conjugacy}\]
which shows that the characteristic polynomial is a well defined invariant. Tracing through the definition in terms of the value of $F = \eta(\sigma_q)$ for $\eta \in H^1(\mathbb Z\sigma_q,\GL_r(\Lambda))$, $p_{n,v}(y)$ has the following explicit formula. For $v \in \mathbb Z_\ell^b$, we denote $F(t_1 = \zeta_{\ell^n}^{v_1},\dots,t_b = \zeta_{\ell^n}^{v_b})$ by $F(\zeta_{\ell^n}^v)$ and define
\begin{equation}\label{eqn: A_n(v)}
    A_n(v) \coloneqq F(\zeta_{\ell^n}^{Q^{k_n(v)-1}v})\dots F(\zeta_{\ell^n}^v) = \prod_{i=1}^{k_n}F(\zeta_{\ell^n}^{Q^{-i}v})
\end{equation}
where we implicitly use that $Q^{k_n(v)}v \equiv v \pmod{\ell^n}$ for the second equality. The characteristic polynomial of $A_n(v)$ is exactly \[p_{n,v}(y) = \det(\Id - yA_n(v)).\] Equivalently, it is the characteristic polynomial of $\sigma_q^{k_n(v)}$ acting on $\Lambda_n(v)$.

As the main results of this section, we will prove two $\ell$-adic convergence results regarding the sequence of polynomials $p_{n,v}(y)$ as $n \to \infty$.

\begin{theorem}\label{thm: Main theorem, abstract, scalar}
Suppose that $Q = q\Id$ is a scalar matrix. For $n$ sufficiently large so that $k_{n+1} = \ell k_n$, the characteristic polynomials satisfy the congruence
\[p_{n+1,v}(y) \equiv p_{n,v}(y) \pmod{k_{n+1}}.\]
\end{theorem}

\begin{remark}\label{rmk: failure of convergence, example}Unfortunately, this strong congruence is not true in general if $Q$ is not scalar (even if $r=1$) as the following example shows. Take $\ell = 5,q_1 = 6, q_2 = 11$ and let $Q$ be the diagonal matrix with entries $q_1,q_2$. Take $F(t_1,t_2) = 1 + t_1^3t_2$ and $v = (1,1) \in \mathbb Z_\ell^2$. Computation shows that $A_3(v) = 49, A_2(v) = 7$ so that the difference is only divisible by $7$ and not $k_3 = 49$ as the above theorem would suggest. Nevertheless, the computational evidence also suggests that the $A_n(v)$ still converge, just with a slower rate of convergence. As we will see in Remark \ref{rmk: failure in general case, sum of rou}, this will be related to the vanishing of certain sums of roots of unity.

For our geometric applications, the following statement is sufficient. Recall that $\PP((\mathbb Z/\ell^n\mathbb Z)^b)$ denotes the space of primitive vectors. It is acted upon by $Q$ and we denote a set of representatives for the orbits of $Q^{\mathbb Z}$ acting on $\PP((\mathbb Z/\ell^n\mathbb Z)^b)$ by $\PP((\mathbb Z/\ell^n\mathbb Z)^b)/\sim$. For $v' = Qv$, we note that $p_{n,v} \equiv p_{n,v'}$ so that $p_{n,v}$ is independent of the choice of representative. The following polynomial depends only on the class $\eta$.

\end{remark}

\begin{definition}
With $A_n(v)$ and $p_{n,v}$ as before, define
\[r_n(y) \coloneqq \prod_{v \in \PP((\mathbb Z/\ell^n\mathbb Z)^b)/\sim}p_{n,v}(y^{k_n(v)/k_{n}}).\]
\end{definition}

\begin{theorem}\label{thm: Main theorem, abstract, general}
Let $Q$ be any matrix in the kernel of $\GL_b(\mathbb Z_\ell) \to \GL_b(\mathbb F_\ell)$. For $n \geq n_0$ so that $k_{n+1} = \ell k_n$, We have
\[r_{n+1}(y) \equiv r_n^{\ell^{b-1}}(y) \pmod{\md}.\]
If $Q = q\Id$, we have the stronger congruence
\[r_{n+1}(y) \equiv r_n^{\ell^{b-1}}(y) \pmod{\ell^{nb}}.\]
\end{theorem}

\begin{remark}
When $b=1$, the two bounds agree since all matrices are scalar! Note that Theorem \ref{thm: Main theorem, abstract, scalar} only implies the following weaker congruence for $b=1$:
\[r_{n+1}(y) \equiv r_n(y) \pmod{k_{n+1}}.\]

\end{remark}

\begin{remark}\label{rmk: sharp congruence}
Numerical evidence shows that these congruences are in fact sharp and the bounds in Theorems \ref{thm: Main theorem, abstract, scalar} and \ref{thm: Main theorem, abstract, general} are realized in most cases (but not always!). For instance, with $r=1,b=2,\ell=3$ and $Q = (1+\ell^2)\Id$ a scalar matrix, the computation
\[A_3\left(\frac{1}{1-\ell},\frac{1}{1-\ell}\right)-A_2\left(\frac{1}{1-\ell},\frac{1}{1-\ell}\right) = 70\ell\]
shows the sharpness of Theorem \ref{thm: Main theorem, abstract, scalar}. The same example also shows the sharpness of part $2$ of Theorem \ref{thm: Main theorem, abstract, general}. Let $d\geq 1$ and $\tau_3,\tau_2 \in \mathbb Z_\ell$ so that $r_{3}(y) = 1 - \tau_3y + \dots$ and $r_2^{\ell}(y) = 1 - \tau_2y + \dots$. Then
\[\tau_3 - \ell\tau_2 = 560\ell^4.\]
\end{remark}

Both the theorems above will depend on the following generalization of Fermat's little theorem to matrices to deal with the case when $r \geq 1$. This generalization of Fermat's little theorem can be seen as the degenerate case of Theorem \ref{thm: Main theorem, abstract, scalar} when $F(T_1,\dots,T_b) = F_0$ is constant in the $T_i$.

\subsection{A generalization of Fermat's little theorem to matrices}

In this subsection we state and prove a generalization of Fermat's little theorem to the case of matrices. As noted in the introduction, this generalization was conjectured by Arnold \cite{Arnold} and proved by Zarelua \cite{Zarelua} (and many other following works). Our proof is short and apparently new\footnote{In the course of writing this paper, we found essentially the same proof by \href{https://rjlipton.wpcomstaging.com/2009/08/07/fermats-little-theorem-for-matrices/}{Qiaochu Yuan} in a blog post from 2009.} and therefore we present it here.

\begin{theorem}[Arnold-Zarelua]\label{thm: Arnold-Zarelua}
Let $A \in M_r(\mathbb Z_\ell)$. Then:
\[\tr A^{\ell^{n+1}} \equiv \tr A^{\ell^n} \pmod{\ell^{n+1}}.\]
In fact, we also have
\[\det(1 - xA^{\ell^{n+1}}) \equiv \det(1 - xA^{\ell^{n}}) \pmod{\ell^{n+1}}.\]
\end{theorem}
\begin{proof}
We fix a $n$. Since we are proving a congruence modulo $\ell^{n+1}$, we can replace $A$ by a $r\times r$ matrix with non negative integer entries. Let $G$ be the directed multigraph with adjacency matrix $A$, i.e it has $r$ vertices labelled from $1$ to $r$ and there are $a_{ij}$ many edges from $i$ to $j$.

A closed path of length $n$ on the graph corresponds to a sequence of edges $e_1,\dots,e_{n-1}$ such that the in-vertex of $e_{i+1}$ is the out-vertex of $e_{i}$ and the path starts and ends at the same vertex. The quantity $\tr A^n$ has the graph theoretic interpretation of being the number of closed paths of length $n$ on $G$. 

Now, consider a closed path $P$ of length $\ell^{n+1}$. The cyclic group of order $\ell^{n+1}$ acts on the path by permuting
\[(e_1,\dots,e_{n-1}) \to (e_2,\dots,e_{n-1},e_1).\] 
Since we are working modulo $\ell^{n+1}$, we can ignore those paths $P$ where the orbit by this action has size $\ell^{n+1}$. The remaining paths $P$ are exactly those which are concatenations of $\ell$ copies of a path of length $\ell^n$. These are exactly counted by $\tr(A^{\ell^n})$ and therefore we have shown the required congruence
\[\tr(A^{\ell^{n+1}}) \equiv \tr(A^{\ell^n}) \pmod{\ell^{n+1}}.\]
To prove the corresponding congruence for characteristic polynomials, we use the well known determinant to trace exponential identity (as formal power series in $x$)
\begin{equation}\label{eqn: det to trace}
    \det(1-xB) = \exp\left(-\sum_{d\geq 1}\frac{\tr(B^d)x^d}{d}\right).
\end{equation} 
Let $d = d_0\ell^e$ for $d_0$ co-prime to $\ell$. The congruence above on powers of $A^{d_0}$ then implies that
\[\tr(A^{d\ell^{n+1}}) \equiv \tr(A^{d\ell^{n}}) \pmod{d\ell^{n+1}}.\]
Since $\ell > 2$, $\alpha\equiv \beta \pmod{\ell^n}$ for $n\geq 1$ implies that $\exp(\alpha) \equiv \exp(\beta) \pmod{\ell^n}$:

To see this, let $t \in \mathbb Z_\ell$ such that $\ell^n|t$. We will show that $e^t \equiv 1 \pmod{\ell^n}$. Supposing this, we see that \[\alpha \equiv \beta \pmod{\ell^n} \implies e^{\alpha-\beta} \equiv 1 \pmod{\ell^n} \implies e^\alpha \equiv e^\beta \pmod{\ell^n}\] since $e^{\beta} \in \mathbb Z_\ell[\![x]\!]$ in our case.

To show that $e^t \equiv 1\pmod{\ell^n}$, we argue by cases. The terms appearing in the Taylor expansion of $\exp(t)$ are of the form $t^r/r!$. If $r=1$, then $\ell^n|t$. In general, Legendre's formula shows that $t^r/r$ is divisible by $\ell^{\delta_{n,r}}$ for  $\delta_{n,r} \coloneqq nr - r/(\ell-1)$. For $r\geq 2$, note that \[\delta_{n,r} \geq n \impliedby nr - r/2 - n \geq 0 \iff 2n \geq \frac{r}{r-1} \text{ which is always true for } n \geq 1.\]

We finish our proof now by noting that the congruences on the traces implies (by the exponential identity)
\[\det(1 - xA^{\ell^{n+1}}) \equiv \det(1 - xA^{\ell^{n}}) \pmod{\ell^{n+1}}.\]

\end{proof}

\subsection{A proof of the main congruences}

In this subsection, we prove Theorems \ref{thm: Main theorem, abstract, scalar} and \ref{thm: Main theorem, abstract, general}. It will help to set up some notation and make some easy reductions first.

Recall that $F(T_1,\dots,T_b)$ is a power series in the $T_i$ and to define $A_{n+1}(v)$, we are required to evaluate $F$ at $T_i = \zeta_{\ell^{n+1}}^{v_i} - 1$ (for $i=1,\dots,b$) which is in the maximal ideal
for the local ring $\mathbb Z_\ell[\zeta_{\ell^{n+1}}]$. Since we are interested in a congruence modulo $ k_{n+1}(v)$ (or $k_{n+1}$), we can truncate the $F$ at some finite degree $d$ so that $(\zeta_{\ell^{n+1}}-1)^d \equiv 0 \pmod{k_{n+1}(v)}$ and suppose that it is a polynomial in the $t_i = T_i + 1$ of the form
\[F = \sum_{I \in \mathbb N^b}F_It_1^{i_1}\dots t_b^{t_b}\]
where the $F_I$ are $r\times r$ matrices over $\mathbb Z_\ell$.

Let $\rho\geq 1$ and for a tuple $J = (I_1,\dots,I_\rho) \in \mathbb (N^b)^\rho$, we define \(F_J = \prod_{j=1}^\rho F_{I_j}.\) Using the standard notation $\langle-,-\rangle$ for inner products (and considering $\mathbb N^b \subset \mathbb Z_\ell^b$), we also define the linear form
\[\lambda_J(v) = \sum_{j=1}^\rho\langle I_j, Q^{-j}v\rangle.\]
In terms of this notation, we see that
\[A_{n+1}^d(v) = \sum_{J \in \mathbb (N^b)^{dk_{n+1}(v)}}F_J\zeta_{\ell^{n+1}}^{\lambda_J(v)}\]
where we have implicitly used that $Q^{k_{n+1}(v)}v \equiv v \pmod{\ell^{n+1}}$. We denote cyclic permutations by
\[\tau(J) = (I_2,I_3,\dots,I_\rho,I_1)\]
and if $k_n(v)|\rho$, we note that
\begin{equation}\label{eqn: lambda_j under cyclic permutations}
    \lambda_{\tau J}(v) \equiv \lambda_J(Qv) \pmod{\ell^{n}}.
\end{equation}

\begin{notation}\label{notation: notation for tuples}
We will argue by considering each tuple along with its cyclic permutations. To that end, we fix some notation that we will use repeatedly. Let $K = (I_1,\dots,I_\rho)$ be a tuple of length $\rho$ such that it is non periodic.\footnote{i.e., the tuples $\tau^i J$ are pairwise distinct for $1 \leq i < \rho$.} For any $\delta = r\rho \in \mathbb N$, we define $J_{K}(\delta) = (I_1,\dots,I_\delta) \coloneqq (K,\dots,K)$ to be the tuple of length $\delta$ where $K$ is concatenated to itself $r$ times. We suppose that $r = r_0\ell^s$ with $r_0$ co-prime to $\ell$.
\end{notation}

We need one more lemma (which will in fact control the rate of congruence) before the proof of Theorem \ref{thm: Main theorem, abstract, scalar}. 

\begin{lemma}\label{lem: sum of q-conj is zero}
For $n\geq 0$, suppose $\rho$ is an integer multiple of $k_n$. For any $w\in\mathbb Z_\ell$,
\[S_{\rho,n}(w) \coloneqq \sum_{i=1}^{\rho}\zeta_{\ell^n}^{q^iw} \equiv 0 \pmod{\rho}.\]
\end{lemma}
\begin{proof}
Let $w = \ell^mw_0$ with $w_0$ a unit. Since $q^{k_n} \equiv 1 \pmod{\ell^n}$ and $\rho/k_n \in \mathbb Z$, we see that
\[S_{\rho,n}(w) = \sum_{i=1}^{\rho}\zeta_{\ell^{n-m}}^{q^iw_0} = \frac{\rho}{k_{n-m}}S_{k_{n-m},n-m}(w_0)\]
where we use the convention that $\zeta_{-m} = 1$ if $m\geq 0$. Therefore, we can suppose that $w$ is a unit and $\rho = k_n$ without loss of generality. Let $\log q = \ell^\alpha x$ with $x$ a unit so that $q^i-1 = i\ell^\alpha x \pmod{\ell^{\alpha+1}}$. We now have two cases to consider. Either $\alpha \geq n$ in which case $\zeta_{\ell^n}^{q^iw} = \zeta_{\ell^n}^w$ and
\[S_{\rho,n}(w) = \rho\zeta_{\ell^n}^w \equiv 0 \pmod{\rho}\]
or $\alpha < n$. In this second case, note that the $\zeta_{\ell^n}^{q^iw}$ are all pairwise distinct for $i\leq k_n = \ell^{n-\alpha}$: 

If $1 \leq j < i \leq \ell^{n-\alpha}$, then \[i-j  < \ell^{n-\alpha} \implies \zeta_{\ell^{n}}^{(q^i-q^j)w} = \zeta_{\ell^{n}}^{(i-j)w\ell^\alpha x + \dots} \neq 1.\]
In fact, the $\zeta_{\ell^n}^{q^iw}$ are a complete set of roots for the polynomial $z^{\ell^{n-\alpha}} = \zeta_{\ell^\alpha}^w$ and $S_{\rho,n}(w)$ is equal to the linear term of this polynomial which is $0$ thus completing the proof.

\end{proof}

\begin{proof}[Proof of Theorem \ref{thm: Main theorem, abstract, scalar}]
For this proof, recall that $Q = q\Id$ is a scalar matrix so that $k_n(v) = k_n$ for all primitive $v$. We reduce the congruence on the characteristic polynomials $p_{n,v}$ to a congruence on traces using the exponential identity (equation \ref{eqn: det to trace})
\[p_{n,v}(y) = \exp\left(-\sum_{d\geq0}\tr(A_n^d(v))\frac{y^d}{d}\right)\]
as in the proof of Theorem \ref{thm: Arnold-Zarelua}. Upon fixing $n$ such that $k_{n+1} = \ell k_n$,it suffices to show the congruence
\[t_n \coloneqq \tr(A_{n+1}^d(v)) - \tr(A_n^d(v)) \equiv 0 \pmod{dk_{n+1}}.\]
We will consider the contributions to $t_n$ from each tuple and its cyclic permutations. In the notation of \ref{notation: notation for tuples}, we take $\delta = dk_{n+1}$ and $J = J_{K}(dk_{n+1})$ and if $\ell|r$, $J_0 = J_K(dk_n)$. Note that
\begin{equation*}
    \lambda_J(v) = \sum_{i=1}^{\rho}\langle I_i,q^{-i}\left(1 + q^{-\rho} + \dots + q^{-(r-1)\rho}\right)v \rangle = \frac{q^{-r\rho}-1}{q^{-\rho}-1}\lambda_K(v).
\end{equation*}

Since $q^i-1 = i\log(q) + \frac{1}{2}(i\log(q))^2 + \dots$ is exactly divisible by $i\log(q)$, there exists some $w \in \mathbb Z_\ell$ so that
\begin{equation}\label{eqn: lambda_j in terms of lambda_k}
    \frac{q^{-r\rho}-1}{q^{-\rho}-1} = \frac{\ell^s\rho}{\rho}w = \ell^sw \implies \lambda_J(v) = \ell^sw\lambda_K(v).
\end{equation}
Moreover, there exists some $y\in\mathbb Z_\ell$ so that $q^{-dk_n} \equiv 1 + \ell^ny \pmod{\ell^{n+1}}$ and therefore
\begin{equation}\label{eqn: lambda_j in terms of lambda_j0}
    \sum_{i=1}^{\ell-1}q^{-idk_n} \equiv \ell + \ell^ny\sum_{i=0}^{\ell-1}i \equiv \ell + \ell^{n+1}y\frac{\ell-1}{2} \pmod{\ell^{n+1}} \implies \lambda_J(v) \equiv \ell\lambda_{J_0}(v) \pmod{\ell^{n+1}}.
\end{equation}

We now have to consider two cases.

\textit{First, suppose $s=0$.} In this case, the only contributions from tuples that are repetitions of $K$ and its cyclic permutations comes from $\tr(A_{n+1}^d(v))$ and is of the form
\[\sum_{i=1}^\rho \tr(F_{\tau^iJ})\zeta_{\ell^{n+1}}^{\lambda_{\tau^iJ}(v)} = \tr(F_J)\sum_{i=1}^{\rho}\zeta_{\ell^{n+1}}^{\lambda_J(q^iv)} = \tr(F_J)\sum_{i=1}^{\rho}\zeta_{\ell^{n+1}}^{q^iw\lambda_K(v)}\]
where for the first equality, we use that the trace is invariant under cyclic permutations and equation \ref{eqn: lambda_j under cyclic permutations} while for the second equality, we use equation \ref{eqn: lambda_j in terms of lambda_k} above and that $s=0$ by assumption. Now, $r\rho = dk_{n+1}$ and since $r$ is co-prime to $\ell$, $k_{n+1}$ being a $\ell$-power necessarily divides $\rho$. In fact, $\rho$ and $dk_{n+1}$ have the same $\ell$-adic valuation. Thus, we can apply Lemma \ref{lem: sum of q-conj is zero} to conclude
\[\sum_{i=1}^\rho \tr(F_{\tau^iJ})\zeta_{\ell^{n+1}}^{\lambda_{\tau^iJ}(v)} \equiv 0 \pmod{\rho} \iff \sum_{i=1}^\rho \tr(F_{\tau^iJ})\zeta_{\ell^{n+1}}^{\lambda_{\tau^iJ}(v)} \equiv 0 \pmod{dk_{n+1}}.\]
\textit{Next, suppose $s>0$.} In this case, we will have contributions from both $\tr(A_{n+1}^d(v))$ and $\tr(A_n^d(v))$ and they are of the form
\begin{align*}
    \sum_{i=1}^\rho\tr(F_{\tau^iJ})\zeta_{\ell^{n+1}}^{\lambda_{\tau^iJ}(v)} - \sum_{i=1}^\rho\tr(F_{\tau^iJ_0})\zeta_{\ell^{n}}^{\lambda_{\tau^iJ_0}(v)} &= \left(\tr(F_K^r) - \tr(F_K^{r/\ell})\right)\sum_{i=1}^\rho\zeta_{\ell^n}^{\lambda_{J}(q^{i}v)}\\ &= \left(\tr(F_K^r) - \tr(F_K^{r/\ell})\right)\sum_{i=1}^\rho\zeta_{\ell^{n+1}}^{q^i\ell^sw\lambda_K(v)}\\
    &\equiv 0 \pmod{r\rho = dk_{n+1}}
\end{align*}
where the first equality follows from invariance of trace under cyclic permutations and equation \ref{eqn: lambda_j in terms of lambda_j0} while the second equation follows from equation \ref{eqn: lambda_j in terms of lambda_k}. For the last congruence, Theorem \ref{thm: Arnold-Zarelua} implies that
\[\tr(F_K^r) - \tr(F_K^{r/\ell}) \equiv 0 \pmod{r}.\]
Moreover, since $dk_{n+1} = r\rho$, we see that $\rho$ is divisible by $dk_{n+1}\ell^{-s}$ and in particular by $k_{n+1-s}$. Therefore, we can apply Lemma \ref{lem: sum of q-conj is zero} to conclude \(\sum_{i=1}^\rho\zeta_{\ell^{n+1-s}}^{q^iw\lambda_K(v)} \equiv 0 \pmod{\rho}.\)

\end{proof}

\begin{remark}\label{rmk: failure in general case, sum of rou}
We remark that the failure of this proof for the general case (see Remark \ref{rmk: failure of convergence, example}) happens exactly at Lemma \ref{lem: sum of q-conj is zero}. If $Q$ is not scalar, it is no longer true that $\lambda_J(Qv) = Q\lambda_J(v)$ and consequently, there exist examples (with $\lambda$ a linear form) such that
\[S_n(\lambda;v) \coloneqq \sum_{j=1}^{k_n(v)}\zeta_{\ell^n}^{\lambda(Q^{-j}v)} \not\equiv 0 \pmod{k_n(v)}.\]
Nevertheless, the above proof shows that if the $S_n(\lambda_J;v) \to 0$ as $n \to \infty$, then the characteristic polynomials $p_{n,v}(y)$ will also converge as $n \to \infty$.  If $\lambda_J(\log(Q)v) \neq 0$, a variation of Lemma \ref{lem: sum of q-conj is zero} still applies to $S_n(\lambda_J;v)$. In fact, numerical evidence supports the vanishing of the limit (for $\lambda_J$ an arbitrary linear form) but we do not know how to prove it. \end{remark}

From now on, we again let $Q \equiv \Id \pmod{\ell}$ be a general matrix. We recall some notation before the proof of Theorem \ref{thm: Main theorem, abstract, general}. We let $V = \mathbb Z_\ell^b$ be a free $\mathbb Z_\ell$ module, $V_n = V/\ell^nV$, $\PP(V_n)$ to be the primitive vectors in $V_n$ and $\PP(V_n)/\sim$ to be a set of representatives under the action by $Q$. The characteristic polynomials we are interested in are
\[r_n(y) = \prod_{v \in \PP((\mathbb Z/\ell^n\mathbb Z)^b)/\sim}p_{n,v}(y^{k_n(v)/k_{n}}).\]
We also fix $n$ sufficiently large and define (in the notation of Lemma \ref{lem: order of Q})
\[V_e = \{v \in V: \frac{k_n(v)}{k_n}|\ell^e \iff \beta_v \geq \beta_0-e \iff Xv\equiv 0 \pmod{\ell^{\beta_0-e}}\}\subset V.\]
By the last equivalent condition, we see that $V_e$ is a (non-empty) submodule of $V$. Since $Q$ commutes with $\log(Q)$ and hence also $X$, we see that $Q$ preserves $V_e$. When $Q = q\Id$, $V_e = V$ since $\beta_v = \beta_0$ for all primitive $v$. Also define $V_{e,n}$ to be the image of $V_e$ in $V/\ell^n V$ under the reduction map. Note that, in general, $V_e \not\cong (\mathbb Z/\ell^n\mathbb Z)^c$ for some $c$ and is only a-priori a finite $\mathbb Z/\ell^n\mathbb Z$ module.

So, let $M$ be an arbitrary finite $\mathbb Z_\ell$ module and $n \geq 0$ be the smallest value such that $\ell^n M = 0$.  An element $v\in M$ is said to be primitive (generalizing our usual notion) when $\ell^{n-1}v \neq 0$ and the set of primitive elements is denoted $\PP(M)$. Our two definitions of primitive are compatible in the sense that
\[\PP(V_{e,n}) = \PP(V/\ell^n V)\cap V_{e,n}.\]

We need one more lemma (analogous to Lemma \ref{lem: sum of q-conj is zero} and also the determining factor for the rate of convergence) before the proof of Theorem \ref{thm: Main theorem, abstract, general}. 

\begin{lemma}\label{lem: sum over all vectors is divisible}
Let $M$ be as above with 
\[\chi: M \to \overline{\mathbb Z}_\ell^\times\]
a character. Then, we have the congruence
\[\sum_{v \in \PP(M)}\chi(v) \equiv 0 \pmod{\ell^{n-1}}.\]
If $M = (\mathbb Z_\ell/\ell^n\mathbb Z_\ell)^b$, we have the stronger congruence
\[\sum_{v \in \PP(M)}\chi(v) \equiv 0 \pmod{\ell^{(n-1)b}}.\]
\end{lemma}
\begin{proof}
Let $|M| = m$, note that \(S_M \coloneqq \sum_{v \in M}\chi(v) \equiv 0 \pmod{m}:\) There are two cases to consider. First, if $\chi$ is the trivial character, then $S_M = m$ and the congruence is clear. Second, if $\chi$ is not trivial, we can find some $m_0 \in M$ so that $\chi(m_0) \neq 1$ and \(S_M = \chi(m_0)S_M \implies S_M = 0 \equiv 0 \pmod{m}.\)

Define 
\[N = \{v \in M: \ell^{n-1}v = 0\} \subset M\]
so that $\PP(M) = M - N$. The module $M$ has size at least $\ell^n$ and the module $N$ has size at least $\ell^{n-1}$. Therefore, we have
\[\sum_{v\in\PP(M)}\chi(v) = \sum_{v\in M}\chi(v) - \sum_{w\in N}\chi(w) = S_M - S_N \equiv 0 \pmod{\ell^{n-1}}\]
since $|M| \equiv |N| \equiv 0 \pmod{\ell^{n-1}}$. 

If $M = (\mathbb Z_\ell/\ell^n\mathbb Z_\ell)^b$ so that $N = (\mathbb Z_\ell/\ell^{n-1}\mathbb Z_\ell)^b$, then the above argument shows the stronger congruence
\[\sum_{v\in\PP(M)}\chi(v) = \sum_{v\in M}\chi(v) - \sum_{w\in N}\chi(w) = S_M - S_N \equiv 0 \pmod{\ell^{(n-1)b}}.\]

\end{proof}

We now prove Theorem \ref{thm: Main theorem, abstract, general}, along the same general lines as the proof of Theorem \ref{thm: Main theorem, abstract, scalar}.

\begin{proof}[Proof of Theorem \ref{thm: Main theorem, abstract, general}]

By the exponential identity \ref{eqn: det to trace}, we have
\[r_n(y) = \exp\left(-\sum_{v \in \PP(V/\ell^n V)/\sim}\sum_{f\geq 0}\frac{\tr A_n^f(v)}{f}y^{\frac{fk_n(v)}{k_n}}\right).\]
Let us fix some $d = d_0\ell^e$ (with $d_0$ co-prime to $\ell$) and collect the terms corresponding to $y^d$ so that with
\[C_{d,n} = \sum_{v \in \PP(V_{e,n})/\sim}\frac{k_n(v)}{dk_n}\tr A_n^{\frac{dk_n}{k_n(v)}}(v) \text{, we have } r_n(y) = \exp\left(-\sum_{d \geq 0}C_{d,n}y^d\right).\]
As in the proof of Theorem \ref{thm: Arnold-Zarelua}, the congruence
\[r_{n+1}(y) \equiv r_n^{\ell^{b-1}}(y) \pmod{\md},\]
is reduced to the congruence
\[C_{d,n+1} \equiv \ell^{b-1}C_{d,n} \pmod{\md}.\]
Since a representative $v \in \PP(V_{e,n})/\sim$ represents $k_n(v)$ many vectors in $\PP(V_{e,n})$ and $A_n(Qv)$ is conjugate to $A_n(v)$ so that their powers have the same traces, we can express $C_{d,n}$ as a sum over \textit{all} primitive vectors by
\[C_{d,n} = \sum_{v \in \PP(V_{e,n})}\frac{1}{dk_n}\tr A_n^{\frac{dk_n}{k_n(v)}}(v).\]
Therefore we are reduced to proving the congruence
\[t_n \coloneqq \sum_{v \in \PP(V_{e,n+1})}\frac{1}{dk_{n+1}}\tr A_{n+1}^{\frac{dk_{n+1}}{k_{n+1}(v)}}(v)-\sum_{v \in \PP(V_{e,n})}\frac{\ell^{b}}{dk_{n+1}}\tr A_n^{\frac{dk_n}{k_n(v)}}(v) \equiv 0 \pmod{\md}\]
where we have implicitly used the assumption that $n$ is sufficiently large so that $k_{n+1} = \ell k_n$. Since every vector in $\PP(V_{e,n})$ has $\ell^b$ many lifts to $\PP(V_{e,n+1})$, we also have
\[t_n = \frac{1}{dk_{n+1}}\sum_{v \in \PP(V_{e,n+1})}\left(\tr A_{n+1}^{\frac{dk_{n+1}}{k_{n+1}(v)}}(v) -\tr A_{n}^{\frac{dk_n}{k_n(v)}}(v) \right).\]
Note that in the expansion
\[\tr A_n^{\frac{dk_n}{k_n(v)}}(v) = \sum_{J \in (\mathbb N^b)^{dk_n}}\tr(F_J)\zeta_{\ell^n}^{\lambda_J(v)},\]
the tuples all have size $dk_n$ independent of $v$. As before, we will argue by fixing a tuple $K$ and considering the contributions from tuples that are multiples of $K$ and their cyclic permutations. In the notation of \ref{notation: notation for tuples}, let $J = J_{K}(dk_{n+1})$ and when $\ell|r$, $J_0 = J_K(dk_n)$.

\textit{First, we suppose that $s = 0$}. In this case, the only contribution to $t_n$ from $K$ will be through $J$ and will be of the form
\[\frac{1}{dk_{n+1}}\sum_{v \in \PP(V_{e,n+1})}\tr(F_J)\zeta_{\ell^{n+1}}^{\lambda_J(v)}.\]
We note that $\zeta_{\ell^{n+1}}^{\lambda_J(v)}$ is a character on $V_{e,n+1}$ and therefore, by Lemma \ref{lem: sum over all vectors is divisible}, there exists some $T_{\lambda_J} \in \mathbb Z_\ell$ such that
\[\frac{1}{dk_{n+1}}\sum_{v \in \PP(V_{e,n+1})}\tr(F_J)\zeta_{\ell^{n+1}}^{\lambda_J(v)} = \frac{\ell^{n}}{dk_{n+1}}T_{\lambda_J}.\]
Moreover, for any cyclic permutation $\tau^iJ$ of $J$, the corresponding contribution is of the same form as before since $Q^i$ permutes $\PP(V_{e,n+1})$:
\[\frac{1}{dk_{n+1}}\sum_{v \in \PP(V_{e,n+1})}\tr(F_{\tau^i J})\zeta_{\ell^{n+1}}^{\lambda_J(Q^iv)} = \frac{1}{dk_{n+1}}\sum_{v \in \PP(V_{e,n+1})}\tr(F_J)\zeta_{\ell^{n+1}}^{\lambda_J(v)} = \frac{\ell^{n}}{dk_{n+1}}T_{\lambda_J}.\]
Therefore, the contribution from all the cyclic permutations of $J$ is together equal to
\[\frac{\rho\ell^{n}}{dk_{n+1}}\tr(F_J)T_{\lambda_J} \equiv 0 \pmod{\ell^n}\]
since the $\ell$-adic valuation of $\rho$ is equal to the $\ell$-adic valuation of $dk_{n+1}$.

\textit{Next, suppose $s>0$}. In this case, the contribution from $K$ will be through $J$ and $J_0$. Since $v \in V_e$, $dk_{n+1}$ is divisible by $k_{n+1}(v)$ so that $Q^{-idk_n}v = v + i\ell^nYv$ for some $Y \in M_b(\mathbb Z_\ell)$ and 
\[\left(\Id + Q^{-dk_n} + \dots + Q^{-(\ell-1)dk_n}\right)v = \ell v + i\ell^n\sum_{i=0}^{\ell-1}Yv = \ell v + \ell^{n+1}\frac{\ell-1}{2}Yv \equiv \ell v \pmod{\ell^{n+1}}.\]
This implies that
\[    \lambda_J(v) = \sum_{i=1}^{dk_{n}}\langle I_i,q^{-i}\left(1 + Q^{-dk_n} + \dots + Q^{-(\ell-1)dk_n}\right)v \rangle   \equiv \ell\lambda_{J_0}(v) \pmod{\ell^{n+1}}\]
which is equivalent to $\zeta_{\ell^{n+1}}^{\lambda_J(v)} = \zeta_{\ell^n}^{\lambda_{J_0}(v)}$. Therefore, the contribution from $J,J_0$ in $t_n$ is of the form
\[\frac{1}{dk_{n+1}}\left(\tr(F_k^r)-\tr(F_k^{r/\ell})\right)\sum_{v\in \PP(V_{e,n+1})}\zeta_{\ell^{n+1}}^{\lambda_{J}(v)} = \frac{\ell^n}{dk_{n+1}}\left(\tr(F_k^r)-\tr(F_k^{r/\ell})\right)T_{\lambda_J}.\]
As above, the cyclic permutations of $K$ give rise to exactly the same contribution so that the total contribution from all cyclic permutations of $K$ is
\[\frac{\rho\ell^{n}}{dk_{n+1}}\left(\tr(F_k^r)-\tr(F_k^{r/\ell})\right)T_{\lambda_J} \equiv 0 \pmod{\ell^{n}}\]
since $\left(\tr(F_k^r)-\tr(F_k^{r/\ell})\right)$ is divisible by $r$ by Theorem \ref{thm: Arnold-Zarelua} and $r\rho = dk_{n+1}$.

\textbf{When $Q = q\Id$}, the proof is exactly the same as above except that we have the stronger congruence
\[\sum_{v \in \PP(V_{e,n+1})}\tr(F_J)\zeta_{\ell^{n+1}}^{\lambda_J(v)} \equiv 0 \pmod{\ell^{nb}}.\]
This follows from the second part of Lemma \ref{lem: sum over all vectors is divisible} since $V_e = V = \mathbb Z_\ell^b$ in this case and $V_{e,n+1} = (\mathbb Z_\ell/\ell^{n+1}\mathbb Z_\ell)^b$.
\end{proof}

\begin{remark}
As one sees from the proof, the modulus of the congruence in Theorem \ref{thm: Main theorem, abstract, general} depends on the structure of $V_{e,n+1}$.
\end{remark}

\section{Explicit examples}\label{sec: examples}

In this section, we prove that the normalized eigenvalues of the characteristic polynomials $h_{n,v}(x)$ defined in the proof of Theorem \ref{thm: main thm, geom} are independent of $n$ for $n$ sufficiently large in the following two examples:

\begin{itemize}\label{defn: Fermat/A-S curves}
    \item \textit{Fermat Curves}: This is the family of curves defined by the equation 
    \[C_n: x^{\ell^n} + y^{\ell^n} + z^{\ell^n} = 0 \subset \mathbb P^2.\]
    We have maps \[\dots \to C_n \to C_{n-1} \to \dots \to C_1 \cong \mathbb P^1\] with $G_n = \Aut(C_n/C_1) = (\mu_{\ell^n})^2$ and the element $(\zeta_1,\zeta_2)$ acts by $[x:y:z] \to [x\zeta_1:y\zeta_2:z]$.
    \item \textit{Artin-Schreier Curves}: This is the family of curves defined by the projective closure of the equation
    \[C_{n}: y^q - y = x^{\ell^n} \subset \mathbb P^2/\mathbb F_q.\]
    The automorphism group in this case is $G_n = \mathbb F_q\times \mu_{\ell^n}$. An element $(a,\zeta)$ in this group acts on the curve by $(x,y) \to (\zeta x,y+a)$.
\end{itemize}

\begin{remark}
The results of this section work in somewhat greater generality, for instance we don't need to restrict to Fermat or Artin-Schreier curves of degree a power of $\ell$. The results also work for various quotients of these curves such as the superelliptic curves $y^m = x^{\ell^n} + a$.

Since the computations in other cases are exactly analogous, we only deal with the above two cases.

\end{remark}

Throughout this section, we identify characters $\chi: \mu_{\ell^n} \to \overline{\mathbb Z}_\ell$ with vectors $v \in \mathbb Z_\ell$ by $\chi(v): \zeta_{\ell^n} \to \zeta_{\ell^n}^v$. We also fix a compatible family of additive characters $\psi_n: \mathbb F_{q^n} \to \overline{\mathbb Z}_\ell$ that satisfy $\psi_n = \tr(\mathbb F_{q^n}/\mathbb F_q)\circ\psi_1$.

In both of the above families of curves, we can decompose $M_n = H^1_{\text{\'et}}(\overline{C}_n,\mathbb Z_\ell)$ into one dimensional eigenspaces $M_n(\chi)$ indexed by characters $\chi$ of $G_n$. In the Fermat curve case, the characters are naturally indexed by $v \in (\mathbb Z/\ell^n\mathbb Z)^2$ while in the second case, the characters are indexed by $(\psi,v)$ where $\psi$ is an additive character of $\mathbb F_q$ and $v \in \mathbb Z/\ell^n\mathbb Z$.

Given a character $\chi: \mu_{\ell^n} \to \overline{\mathbb Z}_\ell$ and $q \equiv 1 \pmod{\ell^n}$, we can define a multiplicative character of $\mathbb F_q^\times$ since the map $x \to x^{\frac{q-1}{\ell^n}}$ induces a surjection
\[\mathbb F_q^\times \to \mu_{\ell^n}(\mathbb F_q) \cong \mu_{\ell^n}\]
and we compose this surjection with $\chi$. By a slight abuse of notation, we also denote this character by $\chi$.

The following well-known theorem (\cite[Corollary 2.2 and Lemma 2.3]{Katz}) identifies the eigenvalues of the Frobenius $\sigma_q$ on $M(\chi)$ with Gauss and Jacobi sums respectively.

\begin{theorem}\label{thm: Katz}
We assume that $q\equiv 1 \pmod{\ell^n}$.
\begin{itemize}
    \item For the Fermat curves $C_n$, let $\eta = (\chi,\chi_2)$ be a character of $G_n = (\mu_{\ell^n})^2$. The eigenvalues of $\sigma_q$ on the eigenspace $M_n(\eta)$ are given by the Jacobi sum
    \[-J_q(\chi_1,\chi_2) = -\sum_{x \in \mathbb F_q}\chi_1(x)\chi_2(1-x).\]
    \item For the Artin-Schreier curves, let $\eta = (\psi,\chi)$ be a character of $G_n = \mathbb F_q\times\mu_{\ell^n}$. The eigenvalues of $\sigma_q$ on the eigenspace $M_n(\eta)$ are given by the Gauss sums
    \[-g_q(\psi,\chi) = -\sum_{x \in \mathbb F_q}\psi(x)\chi(x).\]
\end{itemize}
\end{theorem}
\begin{proof}
We sketch the proof for completeness. In the case of Fermat curves, we would like to count points on the affine curve $x^{\ell^n} + y^{\ell^n} = -1$ while in the case of Artin-Schreier curves, we would like to count points on $y^q - y = x^{\ell^n}$.

We have the identities
\[\sum_{\chi: \mathbb F_q^{\times} \to \mu_{\ell^n}}\chi(x) = \begin{cases} \ell^n &\text{ if } x = y^{\ell^n}\\0 &\text{ otherwise} \end{cases}\]
and
\[\sum_{\psi: \mathbb F_q \to \mu_q}\psi(x) = \begin{cases}q &\text{ if } x = y^q - y\\ 0 &\text{ otherwise}\end{cases}.\]
We can use these identities to test if an element $x \in \mathbb F_q$ is a $\ell^n$-th power or of the form $y^q - y$ and therefore use it to count points:

For the Fermat Curve, we have
\[C_n(F_q) = \sum_{z+w = -1} \sum_{\chi_1,\chi_2: \mathbb F_q^\times \to \mu_{\ell^n}}\chi_1(x)\chi_2(y) \]
while for Artin-Schreier curves:
\[C_n(\mathbb F_q) = \sum_{z\in \mathbb F_q}\sum_{\psi,\chi}\psi(z)\chi(z).\]
Exchanging the summation, this shows that the point counts on the two curves can be expressed in terms of Jacobi and Gauss sums respectively. Finally, we use the Weil-conjectures to identify eigenvalues of the Frobenius action with Jacobi/Gauss sums by varying over all powers of $q$.

\end{proof}

Let us return to the set-up of Theorem \ref{thm: main thm, geom}. The roots of the characteristic polynomial $h_{n,v}(x)$ therefore correspond to $\left(-J_q(\chi_1,\chi_2)\right)^{k_n} = -J_{q^{k_n}}(\chi_1,\chi_2)$  with $v$ corresponding to the character $\chi_1,\chi_2$ and similarly for the Gauss sum in the two cases we are interested in. Put another way, we choose the minimal $q$ so that $q-1$ is exactly divisible by $\ell^n$ and we are looking for a relation between these values for varying $n$.

Luckily, the exact statement we need is a result of Coleman \cite{Coleman} proved using the p-adic Gamma function of Gross-Koblitz \cite{Gross-Koblitz}. Stated in our notation and specialized to our needs, \cite[Theorem 11]{Coleman} takes the following form:

\begin{theorem}[Coleman]\label{thm: Coleman}
Let $v \in \mathbb Z_\ell$, $q=p^f$ be such that $\ell^n$ exactly divides $q-1$. In the notation of the previous theorem, we have
\[ g_{q^\ell}(\psi,\chi_{q^\ell}(v)) = g_q(\psi,\chi_q(v))\chi_q(v)(\ell)c_q\]
for $c_q = c_p^f$ and $c_p = (-1)^{r}p^{\frac{\ell-1}{2}}$ where $r$ depends only on $\ell$.
\end{theorem}
\begin{proof}
In Theorem $11$ of loc. cit., take $b = v/\ell^{n+1}$, $d=\ell$. Note that there is exactly one orbit of size $\ell$ and $c = \left(\sqrt{-p}^{\ell-1}\phi_d(0)\right)^f$, $r = r_\ell + (\ell-1)/2$ in the notation of that paper.
\end{proof}

The following theorem is an immediate consequence of Coleman's theorem and is the required relation.

\begin{theorem}\label{thm: char poly stabilize}
Suppose that $q$ is such that $\ell^n$ exactly divides $q-1$. 
Let $v_1,v_2 \in \mathbb Z_\ell$, $\chi_{q^m}(v_i)$ multiplicative characters of $\mu_{\ell^\infty}\left(\mathbb F_{q^m}\right)$ corresponding to $v_i$ and $\psi_n: \mathbb F_{q^n} \to \overline{\mathbb Z}_\ell$ a compatible series of additive characters as above.

Then, we have the following identities:
\begin{equation}\label{eqn: jacobi sum identity}
    \frac{J_{q}(\chi_q(v_1),\chi_q(v_2))}{q^{1/2}} = \frac{J_{q^\ell}(\chi_{q^\ell}(v_1),\chi_{q^\ell}(v_2))}{q^{\ell/2}}
\end{equation}
and
\begin{equation}\label{eqn: Gauss sum identity}
        \frac{g_{q}(\psi,\chi_q(v))\chi_q(v)(\ell)}{q^{1/2}} = \frac{g_{q^\ell}(\psi,\chi_{q^\ell}(v))}{q^{\ell/2}}
\end{equation}
\end{theorem}

\begin{proof}
We first prove equation \ref{eqn: jacobi sum identity}. We have the well known identity
\[J_q(\chi_1,\chi_2)g_q(\psi,\chi_1\chi_2) = g_q(\psi,\chi_1)g_q(\psi,\chi_2).\]
By Theorem \ref{thm: Coleman}, we then have
\begin{align*}
    J_{q^\ell}(\chi_{q^\ell}(v_1),\chi_{q^\ell}(v_2)) &= \frac{g_{q^\ell}(\psi,\chi_{q^\ell}(v_1))g_{q^\ell}(\psi,\chi_{q^\ell}(v_2))}{g_{q^\ell}(\psi,\chi_1\chi_2)} \\
    &= \frac{g_{q}(\psi,\chi_{q}(v_1))g_{q}(\psi,\chi_{q}(v_2))c_q}{g_{q}(\psi,\chi_1\chi_2)} \\
    &= J_q(\chi_q(v_1),\chi_q(v_2))c_q
\end{align*}
where $q = p^f$. Since $c_q = \pm q^{\frac{\ell-1}{2}}$, we recover equation \ref{eqn: jacobi sum identity} \textit{up to a sign} by dividing by $q^{\ell/2}$. Finally, upon reducing Theorem \ref{thm: Main theorem, abstract, scalar} $\pmod \ell$, we note that the normalized eigenvalues are all congruent $\pmod{\ell}$ and therefore the sign has to be $+1$.

Equation \ref{eqn: Gauss sum identity} follows in exactly the same manner from Theorem \ref{thm: Coleman}.
\end{proof}

\begin{remark}
We note that the above theorem is in exact accord with Case A, Theorem \ref{thm: main thm, geom} since in the notation of that theorem, it shows that the roots of $h_{n+1}(y)$ are equal to the roots of $h_n(y)$. In other words, we not only have a congruence $h_{n+1}(y) \equiv h_n(y) \pmod{\ell^n}$, we have an equality $h_{n+1}(y) = h_n(y)$ in the two cases considered in this section.
\end{remark}

\appendix

%----------------------------------------------------------------------------------------
%	BIBLIOGRAPHY
%----------------------------------------------------------------------------------------

\renewcommand{\refname}{\spacedlowsmallcaps{References}} % For modifying the bibliography heading

\bibliographystyle{alpha}

\bibliography{sample.bib} % The file containing the bibliography

%----------------------------------------------------------------------------------------

\end{document}